\documentclass[11pt]{amsart}

\usepackage[a4paper,hmargin=3.5cm,vmargin=4cm]{geometry}
\usepackage{amsfonts,amssymb,amscd,amstext}
\usepackage{graphicx}
\usepackage[dvips]{epsfig}
\usepackage{color}
\usepackage[utf8]{inputenc}
\usepackage{hyperref}

\usepackage{verbatim}
\usepackage{fancyhdr}
\input xy
\xyoption{all}

\usepackage{times}
\usepackage{enumerate}
\usepackage{titlesec}
\usepackage{mathrsfs}

\titleformat{\section}
{\filcenter\bfseries\large} {\thesection{.}}{0.2cm}{}
\titleformat{\subsection}[runin]
{\bfseries} {\thesubsection{.}}{0.15cm}{}[.]
\titleformat{\subsubsection}[runin]
{\em}{\thesubsubsection{.}}{0.15cm}{}[.]

\usepackage[up,bf]{caption}

\newtheorem{theorem}{Theorem}[section]

\newtheorem{lemma}[theorem]{Lemma}
\newtheorem{corollary}[theorem]{Corollary}

\theoremstyle{definition}
\newtheorem{definition}[theorem]{Definition}
\newtheorem{remark}[theorem]{Remark}

\newtheorem{problem}[theorem]{Problem}

\numberwithin{equation}{section}
\numberwithin{figure}{section}

\newcommand\Ascr{\mathscr{A}}

\newcommand\Cscr{\mathscr{C}}

\newcommand\Hscr{\mathscr{H}}

\newcommand\B{\mathbb{B}}
\newcommand\C{\mathbb{C}}
\newcommand\D{\overline{\mathbb D}}
\newcommand\CP{\mathbb{CP}}
\renewcommand\D{\mathbb D}

\newcommand\N{\mathbb{N}}

\newcommand\R{\mathbb{R}}

\newcommand\T{\mathbb{T}}
\newcommand\Z{\mathbb{Z}}

\newcommand\cd{\overline{\mathbb D}}

\renewcommand\d{\mathbb D}

\newcommand\igot{\mathfrak{i}}

\renewcommand\igot{\mathfrak{i}}

\renewcommand\imath{\igot}

\newcommand\hra{\hookrightarrow}

\newcommand\di{\partial}

\newcommand\dist{\mathrm{dist}}

\newcommand\Id{\mathrm{Id}}

\def\dist{\mathrm{dist}}

\newcommand\Area{\mathrm{Area}}

\begin{document}

\fancyhead[LO]{A properly embedded holomorphic disc in the ball with finite area and dense boundary}
\fancyhead[RE]{F.\ Forstneri\v c} 
\fancyhead[RO,LE]{\thepage}

\thispagestyle{empty}

\vspace*{1cm}
\begin{center}
{\bf\LARGE  A properly embedded holomorphic disc in the ball with finite area and dense boundary curve}  

\vspace*{0.5cm}

{\large\bf  Franc Forstneri{\v c}} 
\end{center}


\vspace*{1cm}

\begin{quote}
{\small
\noindent {\bf Abstract}\hspace*{0.1cm}
In this paper we construct a properly embedded holomorphic disc in the unit ball
$\B^2$ of $\C^2$ having a surprising combination of properties: on the one hand, 
it has finite area and hence is the zero set of a bounded holomorphic function on $\B^2$; 
on the other hand, its boundary curve is everywhere dense in the sphere $b\B^2$.
A similar result is proved in higher dimensions.
Our construction is based on an approximation result in contact geometry, also proved in the paper.

\vspace*{0.2cm}

\noindent{\bf Keywords}\hspace*{0.1cm} holomorphic disc, conformal map, Legendrian curve

\vspace*{0.1cm}

\noindent{\bf MSC (2010):}\hspace*{0.1cm} 32H02; 37J55, 53D10}

\vspace*{0.1cm}
\noindent{\bf Date: \rm April 4, 2018}
\end{quote}

\vspace{0.2cm}


\section{Introduction} 
\label{sec:intro}

In this paper we prove the following result which answers a question posed 
by Filippo Bracci in a private communication. Let $\D=\{z\in\C:|z|<1\}$ be the open unit disc, and let
$\B^n=\{(z_1,\ldots,z_n)\in \C^n : |z|^2 = \sum_{j=1}^n |z_j|^2<1\}$ be the open unit ball in the
complex Euclidean space $\C^n$ for any $n\in \N=\{1,2,3,\ldots\}$.

\begin{theorem}\label{th:main}
For every  $n>1$ and $\epsilon>0$ there exists a proper holomorphic embedding $F\colon \D\hra \B^n$ 
which extends to an injective holomorphic immersion $F\colon \overline \D\setminus\{\pm 1\}\to\overline \B^n$
such that $\Area(F(\D))<\epsilon$ and the boundary $F(b\D\setminus \{\pm1\})$ is everywhere dense in $b\B^n$.
\end{theorem}

Our construction provides an injective holomorphic immersion $F$
of an open neighborhood $U\subset \C$ of $\cd\setminus \{\pm1\}$ into $\C^n$ which is transverse to 
the sphere $b\B^n$ and satisfies $F(U)\cap \B^n=F(\D)$. 
A minor modification of the proof, replacing Lemma \ref{lem:special} by Lemma \ref{lem:special2}, 
yields a map $F$ as above which extends holomorphically across $b\D$ except at one boundary point.
(Clearly it is impossible for $F$ to be smooth on all of $\cd$.)
The result seems especially interesting in dimension $n=2$ since
an embedded holomorphic disc of finite area in the ball $\B^2$ is the zero set of a bounded
holomorphic function on $\B^2$ according to Berndtsson \cite[Theorem 1.1]{Berndtsson1980},
so we have the following corollary to Theorem \ref{th:main}.

\begin{corollary}\label{cor:main}
There is a bounded holomorphic function on $\B^2$ whose zero set
is  a smooth complex curve of finite area, biholomorphic to the disc, with injectively immersed
boundary curve that is everywhere dense in the sphere $b\B^2$.
\end{corollary}

It would be interesting to know whether there 
exists a holomorphic fibration $\B^2\to\D$ by discs as 
in Theorem \ref{th:main}. As Bracci pointed out, nonexistence 
of such a fibration  would lead to an analytic proof of the theorem of
Koziarz and Mok \cite{KoziarzMok2010} that there is no fibration of the ball over the disc which is 
invariant under the action of a co-compact group of automorphisms of the ball.

In the literature there are only few known constructions of {\em properly embedded} holomorphic discs
 with interesting global properties in the $2$-ball, or in any $2$-dimensional manifold. 
A recent one, due to Alarc\'on, Globevnik and L{\'o}pez 
\cite{AlarconGlobevnikLopez2016Crelle}, gives a {\em complete} properly embedded 
holomorphic disc in $\B^2$; however, such discs necessarily have infinite area.
On the other hand, it was shown by Globevnik and Stout in 1986 \cite[Theorem VI.1]{GlobevnikStout1986}
that every strongly pseudoconvex domain $D\subset \C^n$ $(n\ge 2)$ with real
analytic boundary $bD$ contains a proper holomorphic disc $F\colon \D\to D$ (not necessarily embedded)
of arbitrarily small area such that $\overline{F(\D)}= F(\D)\cup\overline  \omega$, where $\omega$ is a given 
nonempty connected subset of $bD$. (It is easy to achieve the latter improvement also in our result.)
The main new point here is that we find properly {\em embedded} holomorphic discs with these properties,
even in the lowest dimensional case $n=2$.
The principal difficulty is that double points of a complex curve in 
a complex surface are stable under deformations, and there are no known constructive methods of removing them. 
One of our main new tools is a deformation procedure, based on transversality and the use of certain 
special harmonic functions, to ensure that double points never appear 
during the construction; see Lemma \ref{lem:position}. This makes our construction
considerably more subtle than the one in \cite{GlobevnikStout1986}.
We also use a more precise version of the exposing of boundary points  
\cite[Lemma 2.1]{ForstnericWold2009}; see Lemma \ref{lem:special}. The mentioned
result in \cite{ForstnericWold2009} has been extended to strongly pseudoconvex domains in higher dimension 
(see Diederich, Forn{\ae}ss and Wold \cite{DiederichFornaessWold2014}) and was used in the
study of the squeezing function and the boundary behaviour of intrinsic metrics;
see Zhang \cite{Zhang2017} and the references therein. 

The construction begins with a new result in contact geometry.  
The sphere $b\B^{n}=S^{2n-1}$ for $n>1$ carries the contact structure $\xi$ 
given by the distribution of complex tangent hyperplanes.
The complement of a point in $b\B^{n}$ is contactomorphic  to the Euclidean space $\R^{2n-1}_{(x,y,z)}$
with its standard contact structure $\xi_0=\ker(dz+xdy)$, where $x,y\in\R^{n-1}$, $z\in\R$  
and $xdy=\sum_{j=1}^{n} x_jdy_j$. 
A smooth curve $f\colon \R\to b\B^{n}$ is said to be {\em complex tangential}, or $\xi$-{\em Legendrian}, if
$\dot f(t)\in \xi_{f(t)}$ for every $t\in\R$. The following result is proved in Sect.\ \ref{sec:dense}. 

%
%
\begin{theorem}\label{th:approximation}
Let $n>1$. Every continuous map $f_0\colon \R\to\R^{2n-1}$ can be approximated in 
the fine $\Cscr^0$ topology by real analytic injective $\xi_0$-Legendrian  immersions $f\colon \R\hra \R^{2n-1}$. 
\end{theorem}

It is possible that Theorem \ref{th:approximation} holds in every real analytic 
contact manifold $(X,\xi)$. Approximation by immersed (not necessarily injective) 
real analytic Legendrian immersion $\R\hra (X,\xi)$ were obtained by 
Globevnik and Stout \cite[Theorem V.1]{GlobevnikStout1986}. 
Results on approximation of smoothly embedded compact isotropic submanifolds 
by real analytic ones can be found in the monograph by Cieliebak and Eliashberg \cite[Sect.\ 6.7]{CieliebakEliashberg2012};
however, the arguments given there do not seem to apply to noncompact isotropic submanifolds. 

Write $\R_+=\{t\in\R: t\ge 0\}$ and $\R_-=\{t\in\R: t\le 0\}$.
The following is an  immediate corollary to Theorem \ref{th:approximation}.

\begin{corollary}\label{cor:dense}
Let $n>1$. For every nonempty open connected subset $\omega$ of $b\B^{n}$ 
there exists a real analytic injective complex tangential immersion $f\colon \R\hra \omega$ 
such that the cluster set of each of the sets $f(\R_+)$ and $f(\R_-)$ 
equals $\overline\omega$. This holds in particular for $\omega=b\B^n$.
\end{corollary}

The existence of a real analytic complex tangential injective immersion $f\colon \R\to b\B^n$
whose image $\Lambda=f(\R)\subset b\B^{n}$ is dense in $b\B^n$ 
is the starting point of our construction. A suitably chosen (thin) complexification of 
$\Lambda$ is an embedded complex disc $\Sigma_0\subset \C^{n}$
of arbitrarily small area such that $\Sigma_0\cap \overline\B^{n} =\Lambda$
(see Lemma \ref{lem:initial-embedding}). By pulling $\Sigma_0$ slightly inside the ball along $\Lambda$
by a suitably chosen holomorphic multiplier, where the amount of pulling diminishes sufficiently fast as we 
approach either of the two ends of $\Lambda$ so that the boundary of $\Sigma_0$ remains in 
the complement of the closed ball, we obtain a properly embedded holomorphic disc $\Sigma$ in $\B^n$ of 
arbitrarily small area whose boundary approximates $\Lambda$ as closely as desired in the fine 
$\Cscr^0$ topology, and hence they are dense in the sphere $b\B^n$.
Since $\Lambda$ is dense in $b\B^n$, it is a rather subtle task to obtain 
injectivity of the limit disc. The main difficulty is that injectivity is not an open condition among 
immersions of noncompact manifolds in any fine topology. (On the other hand, immersions 
form an open set in the fine $\Cscr^1$ topology; see e.g.\ \cite[Sect.\ 2.15]{Narasimhan1985}.) 
We find a disc $F(\D)\subset \B^n$ satisfying Theorem \ref{th:main} as a limit of an inductively 
constructed sequence of properly embedded holomorphic discs $F_k\colon \D\hra \B^n$, where each
$F_k$ is holomorphic on a neighborhood of the closed disc $\overline \D$.
In the induction step we are given a properly embedded complex disc $\Sigma_k=F_k(\D) \subset \B^n$ with smooth boundary
$b\Sigma_k =F_k(b\D) \subset b\B^n$ which intersects the Legendrian curve  $\Lambda=f(\R)\subset b\B^{n}$
transversely at a pair of points $p^\pm_k$. (There may be other intersection points. 
The disc $\Sigma_k$ is actually the intersection of a somewhat bigger embedded holomorphic disc in $\C^n$
with the ball $\B^n$.) Let $E^+_k$ and $E^-_k$ be compact arcs in $\Lambda$ with an 
endpoint $p^+_k$ and $p^+_-$, respectively.
The first step is to find a small perturbation of $\overline \Sigma_k$, fixing the points $p^\pm_k$,
such that the boundary of the new embedded disc intersects the arcs $E^\pm_k$ only at the points $p^\pm_k$
(see Lemma \ref{lem:position}). The next disc $\Sigma_{k+1}$ is then obtained by stretching 
$\Sigma_k$ along the arcs $E^\pm_k$ to the other endpoints
$q^\pm_k$ of $E^\pm_k$ so that the stretched out part lies in a thin tube around 
$E^+_k\cup E^-_k$; see Lemma \ref{lem:special}. The sequence of embedded discs obtained in 
this way converges in the weak $\Cscr^1$ topology, and also in the fine $\Cscr^0$ topology on 
$\overline \D\setminus \{\pm 1\}$, to an embedded disc satisfying Theorem \ref{th:main}.
The details are given in Sect.\ \ref{sec:proof}. 

In conclusion, we mention the following natural question that was raised by a referee.

\begin{problem}
Does Theorem \ref{th:main} hold for discs in an arbitrary strongly pseudoconvex domain 
in $\C^n$, $n\ge 2$, in place of the ball?
\end{problem}

Our methods at several steps crucially depend on having real analytic boundary,
and Theorem \ref{th:approximation} (on Carleman approximation of Legendrian curves by embedded real analytic Legendrian 
images of $\R$) is currently proved only for the standard contact structure on the round sphere in $\C^n$, 
given by the distribution of complex tangent planes.


\section{Densely embedded real analytic Legendrian curves} 
\label{sec:dense}

In this section we prove Theorem \ref{th:approximation}.
Let $n\in\N$. We denote the coordinates on $\R^{2n+1}$ by $(x,y,z)$, where
$x=(x_1,\ldots,x_n)\in\R^n$, $y=(y_1,\ldots,y_n)\in\R^n$, and $z\in \R$.
The standard contact structure $\xi_0=\ker\alpha_0$ on $\R^{2n+1}$ is given by the $1$-form
\begin{equation}\label{eq:alpha0}
	\alpha_0=dz+\sum_{i=1}^n x_idy_i = dz+x dy.
\end{equation}
A smooth immersion $f\colon M\to (\R^{2n+1},\xi_0=\ker\alpha_0)$ from a smooth manifold $M$
is said to be {\em isotropic} if $f^*\alpha_0=0$; an isotropic
immersion is {\em Legendrian} if $M$ has the maximal possible dimension $n$.
(See the monographs by Cieliebak and Eliashberg \cite{CieliebakEliashberg2012} and Geiges \cite{Geiges2008}
for background on contact geometry.) In this paper we only consider maps from $\R$ 
and call them Legendrian irrespectively of the dimension of the  manifold.

Let $k\in\Z_+=\{0,1,2,\ldots\}$. A neighborhood of a $\Cscr^k$ map $f_0\colon \R\to\R^n$ in the {\em fine $\Cscr^k$ topology}
on the space $\Cscr^k(\R,\R^n)$  is of the form
\[
	\bigl\{f\in \Cscr^k(\R,\R^n): |f^{(j)}(t)-f_0^{(j)}(t)| < \epsilon(t)\ \ \forall t\in \R\ \ \forall j=0,\ldots,k\bigr\},
\]
where $\epsilon:\R\to (0,+\infty)$ is a positive continuous function, $f^{(j)}(t)$ denotes the derivative
of order $j$ of $f$ at the point $t\in\R$, and $|\cdotp |$ is the standard Euclidean norm on $\R^n$.
(See Whitney \cite{Whitney1936} or Golubitsky and Guillemin \cite{GolubitskyGuillemin1973} for more information.)

%
%
\begin{proof} [Proof of Theorem \ref{th:approximation}]
Let $(x,y,z)$ be coordinates on $\R^{2n+1}$ as above.
We consider $\R^{2n+1}$ as the standard real subspace of $\C^{2n+1}$ 
and use the same letters to denote complex coordinates on $\C^{2n+1}$.
The contact form $\alpha_0$ on $\R^{2n+1}$ \eqref{eq:alpha0} then extends to a
holomorphic contact form on $\C^{2n+1}$, and a holomorphic map $f\colon D\to \C^{2n+1}$
from a domain $D\subset\C$ will be called Legendrian if $f^*\alpha_0=0$. 

Recall that $\D=\{\zeta \in\C:|\zeta|<1\}$. Let $r>0$. Note that a holomorphic map $f\colon r\D \to  \C^{2n+1}$
is real on the real axis if and only if $f(\zeta) = \overline{f(\bar \zeta)}$ for all $\zeta\in r\D$. 
For any holomorphic map $f$ as above, the symmetrized map
\begin{equation}\label{eq:symmetrize}
	\tilde f(\zeta) = \frac{1}{2} \left(f(\zeta)+\overline{f(\bar \zeta)} \right)
\end{equation}
is holomorphic and real on the real axis.

Let $f_0=(x_0,y_0,z_0) \colon\R\to \R^{2n+1}$ be a given continuous map.
Given a continuous function $\epsilon:\R\to (0,+\infty)$ and a number $\eta_0>0$,
we shall construct a sequence of holomorphic polynomial Legendrian maps $f_j\colon \C \to \C^{2n+1}$ 
and a sequence $\eta_j>0$ such that the following conditions hold for every $j\in\N$, where  (d$_1$) is vacuous:
\begin{itemize}
\item[\rm (a$_j$)] $f_j(\zeta) = \overline {f_j(\bar \zeta)}$ for $\zeta\in \C$,
\vspace{1mm}
\item[\rm (b$_j$)] $f_j\colon [-j,j] \hra\R^{2n+1}$ is a Legendrian embedding,
\vspace{1mm}
\item[\rm (c$_j$)] $|f_j(t)-f_{0}(t)| < \epsilon(t)$ for $t\in [-j,j]$,
\vspace{1mm}
\item[\rm (d$_j$)] $||f_{j}-f_{j-1}||_{\Cscr^1((j-1)\cd)}<\eta_{j-1}$, and 
\vspace{1mm}
\item[\rm (e$_j$)] $0<\eta_j<\eta_{j-1}/2$ and every $\Cscr^1$ map $g\colon [-j,j] \to \R^{2n+1}$ satisfying 
the condition $||g-f_j||_{\Cscr^1( [-j,j])}<2\eta_j$ is an embedding.
\end{itemize}
It is immediate that the sequence $f_j$ converges uniformly on compacts
in $\C$ to a holomorphic Legendrian map $f=\lim_{j\to\infty}f_j \colon \C\to \C^{2n+1}$ whose restriction
to $\R$ is a real analytic Legendrian embedding $f\colon \R\hra\R^{2n+1}$ 
satisfying  $|f(t)-f_0(t)|\le \epsilon(t)$ for all $t\in\R$. 
This will prove Theorem \ref{th:approximation}.

We begin by explaining the base of the induction $(j=1)$. Set 
\[
	\epsilon_1 = \min\{\epsilon(t): -1 \le t\le 1\} >0.
\]
We can find a smooth Legendrian embedding $g \colon [-1,1]\hra \R^{2n+1}$ such that
\begin{equation}\label{eq:estg1}
	|g(t)-f_0(t)|< \epsilon_1/2 \quad \text{for all}\ \ t\in [-1,1].
\end{equation}
(See e.g.\ Geiges \cite[Theorem 3.3.1, p.\ 101]{Geiges2008}, Gromov \cite{Gromov1996PM}, or
\cite[Theorem A.6]{AlarconForstnericLopez2017CM}.)
Let us recall this elementary argument for the case $n=1$, i.e., in $\R^3$. From the contact equation $dz+xdy=0$ we see that 
the third component $z$ of a Legendrian curve $g(t)=(x(t),y(t),z(t))$ for $t\in [-1,1]$ is uniquely determined by the formula
\begin{equation}\label{eq:z-component}
	z(t)=z(0)-\int_0^t x(s)\dot y(s) ds,\qquad t\in [-1,1].
\end{equation}
Hence, a loop $\gamma$ in the Lagrangian $(x,y)$-plane adds a displacement 
for the amount $-\int_\gamma xdy$ to the $z$-variable. By Stokes's theorem, this equals the negative 
of the signed area of the region enclosed by $\gamma$. Hence, it suffices to approximate the $(x,y)$-projection 
of the given continuous arc $f_0\colon [-1,1]\to \R^3$ by a smooth immersed arc containing small  loops whose signed area creates 
a suitable displacement in the $z$-direction, thereby uniformly approximating $f_0$ by a
smooth Legendrian arc $g \colon [-1,1]\to\R^3$. 
Furthermore, by a general position argument (see e.g.\ \cite{GolubitskyGuillemin1973})
we can approximate its Lagrange projection $g_L=(x,y):[-1,1] \to\R^2$ 
in $\Cscr^2([-1,1])$ by a smooth immersion with only simple (transverse) double
points. Note that $g(t_1)=g(t_2)$ for some  pair of numbers $t_1\ne t_2$ if and only if $g_L(t_1)=g_L(t_2)$
and $\int_{t_1}^{t_2} x(s)\dot y(s) ds = 0$. 
Since $g_L$ has at most finitely many double point loops, we can arrange by a generic
$\Cscr^2$-small perturbation of $g_L$ away from its double points that the signed area enclosed by any 
of its double point loops is nonzero, thereby ensuring that the new map $g$ is a Legendrian 
embedding and the estimate \eqref{eq:estg1} still holds.
Furthermore, we see from \eqref{eq:z-component} that $\Cscr^2$ approximation
of the Lagrange projection $g_L$ gives $\Cscr^1$ approximation of the last component $z$.
A similar argument applies in any dimension. 

Let $g$ be as above, satisfying \eqref{eq:estg1}.
Pick a number $\delta$ with $0<\delta < \epsilon_1/2$. We claim that there is a 
polynomial Legendrian map $f_1=(x_1,y_1,z_1) \colon \C\to\C^{2n+1}$ satisfying 
$f_1(\zeta)= \overline{f_1(\bar \zeta)}$ and 
\begin{equation}\label{eq:estg3}
	\|f_1 -  g\|_{\Cscr^1([-1,1])} < \delta < \epsilon_1/2.
\end{equation}
To find such $f_1$, we  apply  Weierstrass's  theorem to 
approximate the Lagrange projection $g_L=(x,y)\colon [-1,+1] \to \R^{2n}$ of $g$ 
in $\Cscr^2([-1,1])$ by a holomorphic polynomial map $(x_1,y_1)\colon \C\to \C^{2n}$ 
which is real on the real axis (the last condition
is easily ensured by replacing the approximating map with its symmetrization, see \eqref{eq:symmetrize}).
We then obtain the last component $z_1$ of $f_1$ by integration as in \eqref{eq:z-component}:
\[
	z_{1}(\zeta)= z_0(0)-\int_0^\zeta x_{1}(t)\dot y_{1}(t)\, dt.
\]
Note that $z_1$ is then also real on the real axis. If $\delta>0$ is chosen small enough,
then $f_1\colon [-1,1]\to \R^{2n+1}\subset \C^{2n+1}$ is a Legendrian embedding.
Thus, conditions (a$_1$) and (b$_1$) hold, and the inequalities \eqref{eq:estg1} and  \eqref{eq:estg3} 
yield (c$_1$). Condition (d$_1$) is vacuous. Pick a constant $\eta_1>0$ such that condition  (e$_1$) holds. 
This provides the base of the induction.

%
%
Assume that for some $j\in \N$ we have found maps $f_1,\ldots,f_j$
and numbers $\eta_1,\ldots,\eta_j$ satisfying conditions (a$_k$)--(e$_k$)
for $k=1,\ldots,j$.  In particular, there is a number $0<\sigma<1$ such that 
$f_j\colon [-j-\sigma,j+\sigma] \hra \R^{2n+1}$  is a Legendrian embedding. Set 
$E_j=j\cd\, \cup [-j-1,j+1] \subset\C$. After decreasing $\sigma>0$ if necessary, the same arguments 
as in the initial step furnish a map $\tilde f_j\colon (j+\sigma)\D  \cup  [-j-1,j+1] \to \C^{2n+1}$
which equals $f_j$ on $(j+\sigma)\d$ and such that $f_j\colon [-j-1,j+1] \hra \R^{2n+1}$ 
is a smooth Legendrian embedding satisfying 
\[
	|\tilde f_j(t)-f_0(t)|<\epsilon(t),\qquad t\in [-j-1,j+1].
\]
Write $\tilde f_j=(\tilde x_j,\tilde y_j,\tilde z_j)$. We apply 
Mergelyan's approximation theorem and the symmetrization argument (cf.\ \eqref{eq:symmetrize})
in order to find a holomorphic polynomial map $(x_{j+1},y_{j+1})\colon  \C\to \C^{2n}$ which is real on $\R$ 
and approximates  the map $(\tilde x_j,\tilde y_j)$ as closely as desired in $\Cscr^2(E_j)$. Setting
\[
	z_{j+1}(\zeta)= z_0(0)-\int_0^\zeta x_{j+1}(t)\dot y_{j+1}(t)\, dt,\quad \zeta\in\C,
\]
we get a polynomial Legendrian map $f_{j+1}=(x_{j+1},y_{j+1},z_{j+1})\colon \C\to \C^{2n+1}$ 
which is real on $\R$ and approximates $f_j$  in $\Cscr^1(E_j)$. 
If the approximation is close enough then $f_{j+1}$  satisfies conditions (a$_{j+1}$)--(d$_{j+1}$).
Finally, pick $\eta_{j+1}>0$ satisfying condition (e$_{j+1}$) and the induction may proceed.
This completes the proof of Theorem \ref{th:approximation}.
\end{proof}
%
%


\section{A general position result} 
\label{sec:generalposition}

In this section we prove a general position result (see Lemma \ref{lem:position})
which is used in the proof of Theorem \ref{th:main}. For simplicity of notation we focus on the case $n=2$,
although the proof carries over to the higher dimensional case $n>2$.  

Recall that $\D=\{z\in\C: |z|<1\}$. Let $\T=b\d=\{z\in \C: |z|=1\}$. For any $k\in \Z_+ \cup\{+\infty\}$ 
and domain $D\subset \C$ with smooth boundary we denote by $\Ascr^k(D)$ the space of 
functions $h\colon \overline D \to\C$ of class $\Cscr^k(\overline D)$ which are holomorphic in $D$, 
and we write $\Ascr^0(D)=\Ascr(D)$. We also introduce the function space
\begin{equation}\label{eq:H}
	\Hscr = \{h=u+\imath v\in \Ascr^\infty(\d) : h(\bar z)=\overline{h(z)},\ \  u|_{\T}=0\ \text{near}\ \pm1\}.
\end{equation}
Note that for every $h=u+\imath v \in \Hscr$ we have $h(\pm 1)=0$, $u(\bar z)=u(z)$ and $v(\bar z)=-v(z)$ 
for all $z\in \cd$; in particular, $v(x)=0$ for all $x\in [-1,1]$. Note that $\Hscr$ is a nonclosed real vector subspace
of $\Ascr^\infty(\d)$. Every function in $\Hscr$ is uniquely determined by a smooth real function 
$u\in \Cscr^\infty(\T)$ supported away from the points $\pm1$ and satisfying $u(e^{\imath t}) = u(e^{-\imath t})$
for all $t\in \R$. Indeed, if $u\colon \cd\to \R$ is the harmonic extension of $u\colon \T\to\R$
and $v$ is the harmonic conjugate of $u$ determined by $v(0)=0$, 
then the function $h=u+\imath v$ belongs to $\Hscr$ and is given by the classical integral formula
\begin{equation}\label{eq:Poisson}
	h(z)= T[u](z) = \frac{1}{2\pi} 
	\int_0^{2\pi} \frac{e^{\imath\theta}+z}{e^{\imath\theta}-z}\, u(e^{\imath\theta}) \, d\theta, \quad z\in \d.
\end{equation}
The real part of the integral operator on the right hand side above is the Poisson integral, while
the imaginary part  is the Hilbert (conjugate function) transform. We write
\begin{equation}\label{eq:Hplus}
	\Hscr^\pm  = \{h=u+\imath v \in \Hscr : \pm u\ge 0\},\quad \Hscr^\pm_* = \Hscr^\pm \setminus\{0\}.
\end{equation}
Note that the sets $\Hscr^\pm$ and $\Hscr^\pm_*$ are cones, i.e., closed under addition and 
multiplication by nonnegative (resp.\ positive) real numbers, and we have $\Hscr^+ \cap \Hscr^- =\{0\}$. Furthermore,
\begin{equation}\label{eq:uatpm1}
	h=u+\imath v\in \Hscr^+_* \ \Longrightarrow \ u>0\ \text{on}\ \d,\ \ \frac{\di u}{\di x}(-1)>0,\ \ \frac{\di u}{\di x}(1)<0
\end{equation}
where the last two inequalities follow from the Hopf lemma (since $u|_{\T}=0$ near $\pm 1$).
The use of these function spaces will become apparent in the proof of Theorem \ref{th:main} in the following section.

Let $\sigma\colon \C^2_*:=\C^2\setminus\{0\} \to\CP^1$ denote the projection onto the Riemann sphere
whose fibres are complex lines through the origin. 

%
%
\begin{lemma}\label{lem:position}
Let $f=(f_1,f_2):\cd\to \C^2_*$ be a map of class $\Ascr^\infty(\d)$ such that
\begin{itemize}
\item[\rm (a)] $|f|^2:=|f_1|^2+|f_2|^2\le 1$ on $(-1,1)=\d\cap \R$, 
\vspace{1mm}
\item[\rm (b)] $|f|>1$ on $\T\setminus \{\pm1\}$, and 
\vspace{1mm}
\item[\rm (c)] $\sigma\circ f\colon \cd\to\CP^1$ is an immersion.
\end{itemize}
Assume that $E\subset b\B^2$ is a smoothly embedded compact curve 
such that $f(\pm 1)\notin E$. Given a number $\eta\in (0,1)$,  there is a function $h\in \Hscr^+_*$ 
arbitrarily close to $0$ in $\Cscr^1(\cd)$ such that the immersion
\begin{equation}\label{eq:fh}
	f_h:=e^{-h}f:\cd\to \C^2_* 
\end{equation}
satisfies the following conditions:
\begin{enumerate} 
\item $|f_h|<1$ on $(-1,1)$ and $|f_h|>(1-\eta)|f|+\eta>1$ on $\T\setminus \{\pm1\}$,
\vspace{1mm}
\item $f_h$ is transverse to $b\B^2$, and   
\vspace{1mm}
\item $f_h(\cd)\cap E=\varnothing$. 
\end{enumerate}
\end{lemma}

\begin{remark}\label{rem:Ch}
The main point to ensure injectivity is to achieve condition (3). By dimension reasons, this is easly 
done by a more general type of perturbation of $f$. However, we use the specific perturbations
obtained by multiplying with a function $e^{-h}$ with $h\in \Hscr_*^+$ in order to be able to
control the very subtle induction process in the proof of Theorem \ref{th:main}.
Note that $\sigma\circ f_h=\sigma\circ f\colon\cd\to \CP^1$ is an immersion by the assumption,
and hence $f_h$ is an immersion. If $f_h$ satisfies the conclusion of the lemma, then the set
\[
	C=\{z\in \cd : f_h(z)\in b\B^2\} \subset \d\cup\{\pm1\}
\]
is a smooth, closed, not necessarily connected curve containing the points $\pm1$, and 
its image $f_h(C) = f_h(\cd) \cap b\B^2$ is a smooth curve  disjoint from $E$. Hence, $C$ and $f_h(C)$ 
are finite unions of pairwise disjoint smooth Jordan curves. Each connected component of
$f_h(C)$ bounds a connected component of the complex curve $f_h(\d)\cap \B^2$
(a properly immersed complex disc in $\B^2$). Since $|f_h|<1$ on $(-1,1)$, there is a component 
$\Omega$ of $\D\setminus C$ containing $(-1,1)$, and $b\Omega\subset \d\cup\{\pm 1\}$ is a closed
Jordan curve containing the points $\pm 1$. This component $\Omega$ will be of main interest 
in the proof of Theorem \ref{th:main}. 
\qed \end{remark}

\begin{proof}
Given $h=u+\imath v\in \Hscr$, we  define the functions $\rho=\rho_0$ and $\rho_h$ by
\begin{equation}\label{eq:def-rho}
	\rho=\log|f|\colon \cd\to \R,\quad \rho_h:= \log |e^{-h}f| = -u + \rho : \cd\to \R.
\end{equation}
Conditions (a) and (b) on $f$ imply that 
\begin{equation}\label{eq:properties-rho}
	\text{$\rho\le 0$ on $[-1,1]=\cd\cap\R$, \ \ \ \ $\rho(\pm 1)=0$, \ \ \ \  $\rho>0$ on $\T\setminus \{\pm1\}$.}
\end{equation}
It is obvious that for any function $h\in \Hscr^+_*$ with sufficiently small $\Cscr^0(\cd)$ norm
and such that the support of $u|_\T=\Re h|_\T$ avoids a certain fixed neighborhood of the points $\pm1$, 
the map $f_h$ given by \eqref{eq:fh} satisfies condition (1) of the lemma.
(Recall that $\Re h|_\T$ vanishes near the points $\pm 1$ by the definition of the space $\Hscr$.)
In particular, for every fixed $h\in \Hscr^+_*$ this holds for the map $f_{th}$ for all small enough $t>0$.

Note that the map $f_h$ \eqref{eq:fh} intersects the sphere $b\B^2$ 
transversely if and only if $0$ is a regular value of the function $\rho_h$ \eqref{eq:def-rho}.
From \eqref{eq:properties-rho} it follows that 
$\frac{\di \rho}{\di x}(-1)\le 0$ and $\frac{\di \rho}{\di x}(1)\ge 0$. Together with \eqref{eq:uatpm1}
we see that for every $h\in \Hscr^+_*$ we have 
\[
	\frac{\di \rho_h}{\di x}(-1) =\frac{\di \rho}{\di x}(-1) - \frac{\di u}{\di x}(-1) <0,\quad
	\frac{\di \rho_h}{\di x}(1)>0.
\]
Replacing $f$ by $e^{-h}f$ for some such $h$ close to $0$, we may assume that $\rho=\log |f|$ satisfies these conditions.
Hence, there are discs $U^\pm\subset \C$ around the points $\pm1$, respectively, such that 
$d\rho\ne 0$ on $\cd\cap (\overline U^+\cup \overline U^-)$. Since this set is compact, 
it follows that for all $h\in \Hscr$ with sufficiently small $\Cscr^1(\cd)$ norm we have that
\begin{equation}\label{eq:drhoh}
	d\rho_h\ne 0\ \ \text{on}\ \  \cd\cap (\overline U^+\cup \overline U^-).
\end{equation}
Furthermore, since the curve $E$ does not contain the points $f(\pm 1)$,
we may choose the discs $U^\pm$ small enough such that 
\begin{equation}\label{eq:fh3}
	f_h\big(\cd \cap (\overline U^+ \cup \overline U^-)\bigr) \cap E =\varnothing
\end{equation}
holds for all $h\in \Hscr$ sufficiently close to $0$ in $\Cscr^1(\cd)$.

Recall that $\rho=\log|f| <0$ on $(-1,1)=\d\,\cap\R$. Hence,  
there is an open set $U_0\Subset \d$ containing the compact interval $(-1,1)\setminus (U^+\cup U^-)\subset\R$
such that $\rho\le -c<0$ on $\overline U_0$ for some constant $c>0$.
Since $\rho>0$ on $\T\setminus \{\pm1\}$, a similar argument gives an open set 
$U_1\Subset \C$ containing the compact set $\T\setminus (U^+\cup U^-)$ 
(the union of two closed circular arcs)  
such that $\rho\ge c'>0$ on $\overline U_1\cap \cd$ for some $c'>0$. 
It follows that for all $h\in \Hscr$ sufficiently close to $0$ in $\Cscr^1(\cd)$ we have that
$\rho_h<0$ on $\overline U_0$, $\rho_h>0$ on $\overline U_1\cap \cd$, and hence
\begin{equation}\label{eq:fh4}           
	f_h \bigl(\cd\cap(\overline U_0\cup \overline U_1)\bigr) \cap b\B^2 =\varnothing.
\end{equation}
For such $h$ it follows in view of \eqref{eq:drhoh} that 
\begin{equation}\label{eq:setK}
	\{z\in\cd: f_h(z)\in b\B^2,\ d\rho_h(z) = 0\} \subset K:= \cd\setminus (U_0\cup U_1\cup U^+\cup U^-).
\end{equation}
Note that the set on the left hand side above is precisely the set of points in $\cd$ at which 
the map $f$ fails to be transverse to the sphere $b\B^2$. The set $K$ is compact and 
contained in $\d \setminus (-1,1)$. Pick $h=u+\imath v \in \Hscr^+_*$ and consider 
the family of functions $\rho_{th}= - tu + \rho$ for $t\in \R$. Since $\di \rho_{th} / \di t = -u<0$ on $\D$, 
transversality theorem (see Abraham \cite{Abraham1963}) 
implies that for a generic choice of $t$, $0$ is a regular value of the function $\rho_{th}|_\D$. 
By choosing $t>0$ small enough and taking into account also \eqref{eq:drhoh} and \eqref{eq:fh4}, 
we infer that the map $f_{th}=e^{-th}f:\cd\to\C^2_*$ is transverse to $b\B^2$ (hence condition (2) holds)
and it also satisfies condition (1). Replacing $f$ by $f_{th}$, we may assume that $f$ 
satisfies conditions (1) and (2).

It remains to achieve also condition (3) in the lemma.
From \eqref{eq:fh3} and \eqref{eq:fh4} it follows that $\{z\in \cd: f_h(z)\in E\}$ 
is contained in the compact set $K\subset \d$ defined in \eqref{eq:setK}.
To conclude the proof, it suffices to find finitely many functions $h_1,\ldots,h_N\in \Hscr^+_*$ such that,
writing $t=(t_1,\ldots,t_N)\in\R^N$, the family of maps 
\begin{equation}\label{eq:ft}
	f_t(z)= \exp\biggl(-\sum_{j=1}^N t_j h_j(z)\biggr) f(z),\quad z\in \cd 
\end{equation}
satisfies $f_t(\cd) \cap E =\varnothing$ for a generic choice of $t\in\R^N$ near $0$. 
By choosing $t=(t_1,\ldots, t_N)\in \R^N$ close enough to $0$ and such that $t_j>0$ for 
all $j=1,\ldots,N$, the map $f_t$ will also satisfy conditions (1) and (2) in the lemma,
thereby completing the proof.

\vspace{1mm}

\noindent {\em Claim:} For every point $z\in \d\setminus (-1,1)$ there exist functions 
$h_1,h_2\in \Hscr^+$ such that the values $h_1(z),\, h_2(z)\in\C$ are $\R$-linearly independent.

\begin{proof}[Proof of the claim]
Let $\delta_\theta$ denote the probability measure on $\T$ representing the 
Dirac mass of the point $e^{\imath \theta}\in \T$.
Choose a sequence of smooth nonnegative even functions $u_j\colon \T \to \R_+$ 
supported near $\pm \imath$ such that $u_j(e^{\imath \theta})=u_j(e^{-\imath \theta})$ 
for all $\theta\in\R$ and 
\[
	\lim_{j\to\infty} \frac{1}{2\pi} u_j  d\theta  = \delta_{\pi/2}+\delta_{-\pi/2}
\]
as measures. From \eqref{eq:Poisson} we have that 
\begin{equation}\label{eq:Tuj}
	\lim_{j\to\infty} T[u_j](z) =  T[\delta_{\pi/2}+\delta_{-\pi/2}](z) 
	= \frac{i+z}{i-z}+\frac{-i+z}{-i-z} 
	= 2\frac{1-|z|^4 - 2\imath \Im (z^2)}{|1+z^2|^2}.
\end{equation}
The imaginary part of this expression vanishes precisely when $\Im(z^2)=0$ which is the union
of the two coordinate axes. Since $u_j$ is an even function, we see that 
$h_j=T[u_j]$ is real on the segment $J=\{\imath y : y\in (-1,1)\}=\D\cap \imath\R$, and 
it is nonvanishing at any given point $z_0=\imath y_0 \in J\setminus \{0\}$ for all big $j\in\N$ as follows from \eqref{eq:Tuj}. 
Let $\phi_a(z)=(z-a)/(1-a z)$ for $a\in (-1,1)$; this is a holomorphic automorphism of the disc 
which maps the interval $[-1,1]$ to itself, and if $a\ne 0$ then $\phi_a(J)\cap J=\varnothing$.
Note that $\phi_a(\bar z) = \overline{\phi_a(z)}$ and $\phi_a(\pm 1) = \pm 1$; hence, 
the precomposition $h\mapsto h\circ \phi_a$ preserves the class $\Hscr^+_*$. 
Choosing $a\in (-1,1)\setminus \{0\}$ we have that $\Im (h_j\circ \phi_a)(z_0)= \Im h_j(\phi_a(z_0))\ne 0$
for $j\in \N$ big enough as is seen from \eqref{eq:Tuj} and the fact that the point $\phi_a(z_0)\in\D$ 
does not lie in the union of the coordinate axes.
This gives two functions in $\Hscr^+_*$, namely $h_j$ and $h_j\circ\phi_a$,
whose values at the given point $z_0=\imath y_0 \in J\setminus \{0\}$ are $\R$-linearly independent.
This establishes the claim for points in $J\setminus \{0\}$, and for other points in $\d\setminus (-1,1)$
we get the same conclusion by precomposing with an automorphisms $\phi_a$, $a\in (-1,1)$.
\end{proof}

Since the set $K$ \eqref{eq:setK} is compact and contained in 
$\d\setminus (-1,1)$, the above claim yields finitely many functions $h_1,\ldots,h_N\in \Hscr^+_*$
such that for every point $z\in K$ the vectors $h_j(z)\in \C$ $(j=1,\ldots,N)$ span $\C$ over $\R$. 
Consider the corresponding family of maps $f_t\colon \cd\to\C^2$ given by \eqref{eq:ft}.
Note that 
\[
	\frac{\di f_t(z)}{\di t_j}\bigg|_{t=0} = -h_j(z) f(z),\quad j=1,\ldots,l.
\]
It follows that the map 
\begin{equation}\label{eq:ftz}
	\cd\times \R^N\ni (z,t)\mapsto f_t(z)\in \C^2
\end{equation} 
is a submersion over $K$ at $t=0$, and hence for all $t\in \R^N$ near $0$.  Indeed, we have 
$\sigma\circ f_t=\sigma\circ f\colon \cd\to \CP^1$ which is an immersion  
by the assumption (c), while for each $z\in K$ the partial differential $\di_t f_t(z)|_{t=0}$ 
is surjective onto the radial direction $\C f(z)= \ker d\sigma_{f(z)}$ by the choice of the functions
$h_1,\ldots, h_N$. Transversality theorem \cite{Abraham1963} implies that for a generic 
$t\in \R^N$ near $0$ the map $f_t|_K \colon K \to\C^2$ misses $E$ by dimension reasons. In view of
\eqref{eq:fh3} and \eqref{eq:fh4} it follows that for any such $t$ we have $f_t(\cd)\cap E=\varnothing$.

The case $n>2$ requires a minor change in the last step of the proof. The map
\eqref{eq:ftz} is now a submersion onto its image which is an immersed complex $2$-dimensional submanifold
of $\C^n$. This submanifold intersects the curve $E$ in a set of finite linear measure, and hence
for a generic choice of $t\in \R^N$ the map $f_t|_K \colon K \to\C^n$ misses $E$ as before.
\end{proof}


\section{A lemma on conformal mappings} 
\label{sec:conformal}

The main results of this section are Lemmas \ref{lem:special} and \ref{lem:special2} on the behaviour of
biholomorphic maps from planar domains onto domains with exposed boundary points.

Let $z=x+\imath y$ denote the coordinate on $\C$. Consider the antiholomorphic involutions
\begin{equation}\label{eq:reflections}
	\tau_x(x+\imath y) = -x + \imath  y,\quad \tau_y(x+\imath y) = x - \imath  y.
\end{equation}
Note that $\tau_x\circ \tau_y=\tau_y\circ \tau_x$ is the reflection $z\mapsto-z$ across the origin $0\in\C$. The 
involutions $\tau_x,\tau_y$ generate an abelian group 
\begin{equation}\label{eq:Gamma}
	\Gamma = \langle \tau_x,\tau_y\rangle \cong\Z_2^2.
\end{equation}
A set $D\subset \C$ is said to be {\em $\Gamma$-invariant} if $\gamma(D)=D$ holds 
for all $\gamma\in\Gamma$. A map $\phi\colon D\to\C$ defined on a $\Gamma$-invariant set is said to be 
{\em $\Gamma$-equivariant} if
\[
	\phi= \gamma\circ \phi \circ\gamma\quad \text{holds for all}\ \ \gamma\in \Gamma.
\]
Since each $\gamma \in \Gamma$ is an involution, this is equivalent to $\gamma\circ \phi=\phi\circ\gamma$. 
A $\Gamma$-equivariant map $\phi\colon D\to\C$ takes $\R\cap D$ into $\R$ and 
$\imath \R\cap D$ into $\imath \R\cap D'$; in particular, $\phi(0)=0$. 
For every map $\phi$ from a $\Gamma$-invariant domain, the map
$\tilde \phi= \frac{1}{4} \sum_{\gamma\in \Gamma} \gamma\circ \phi \circ\gamma$
is $\Gamma$-equivariant.

\begin{definition}\label{def:special}
A nonempty connected domain $D\subset \C$ is {\em special} if it is bounded with $\Cscr^\infty$ smooth boundary, 
simply connected, and $\Gamma$-invariant.
\end{definition}

It is easily seen that a special domain $D$ intersects the real line in an interval
$(-a,a)$ for some $a>0$, and at the points $\pm a$ the boundary $bD$ is tangent 
to the vertical line $x=\pm a$. This interval $(-a,a)$ will be called the {\em base} of $D$. 
The analogous observation holds for the intersection of $D$ with the imaginary axis.
Recall that a biholomorphism between a pair of bounded planar domains with smooth boundaries
extends to a smooth diffeomorphism between their closures in view of the theorems by 
Carath{\'e}odory \cite{Caratheodory1913} and Kellogg \cite{Kellogg1912}. We record the following observation.

\begin{lemma}\label{lem:conformalspecial}
Assume that $D$ is a special domain (Def.\ \ref{def:special}) and $\phi\colon D\to D'$ is a biholomorphic map 
onto a bounded domain $D'=\phi(D)$ with smooth boundary satisfying
\begin{equation}\label{eq:normalize}
	\phi(0)=0\ \ \text{and}\ \ \phi'(0)>0.
\end{equation}
Then $\phi$ is $\Gamma$-equivariant if and only if the domain $D'$ is special.
In particular, a special domain $D\subset \C$ with the base $(-a,a)$ admits a $\Gamma$-equivariant 
biholomorphism $\phi\colon\D\to D$ satisfying \eqref{eq:normalize} and $\phi(\pm 1)=\pm a$.
\end{lemma}

\begin{proof}
Assume that $D'$ is special.  For every $\gamma\in\Gamma$ the map $\gamma \circ \phi\circ \gamma \colon D\to D'$ 
is then a well defined biholomorphism satisfying the normalization \eqref{eq:normalize}, so it equals $\phi$.
This shows that $\phi$ is $\Gamma$-equivariant. The converse is obvious.
\end{proof}

The following {\em exposing of boundary points} lemma is the main result of this section.

%
%
\begin{lemma}\label{lem:special}
Assume that $D\subset \C$ is a special domain with the base $(-a,a)$. 
Fix a number $b>a$ and set $I^+=[a,b]$, $I^-=[-b,-a]$, and $I=I^+\cup I^-$. 
Given an open neighborhood $V\subset\C$ of $I$, a number $\epsilon>0$, and an integer
$k\in\Z_+$ there exists a $\Gamma$-equivariant biholomorphism $\phi\colon D\to \phi(D)=D'$ onto a 
special domain $D'$ with the base $(-b,b)$ satisfying the following conditions:
\begin{itemize}
\item[\rm (a)] $\phi(0)=0$, $\phi'(0)>0$, $\phi(\pm a)=\pm b$, 
\vspace{1mm}
\item[\rm (b)] $D\subset D'\subset D\cup V$ (hence $D \setminus V = D' \setminus V$), and
\vspace{1mm}
\item[\rm (c)] $\|\phi-\Id\|_{\Cscr^k(\overline D\setminus V)}<\epsilon$. 
\end{itemize}
\end{lemma}

\begin{remark}
The main improvement over \cite[Lemma 2.1]{ForstnericWold2009}
is that the domain $D'$ agrees with $D$ outside a thin neighborhood $V$ of the arcs $I=I^+\cup I^-$
(part (b)). Also, the biholomorphic map $\phi\colon D\to \phi(D)=D'$ is $\Gamma$-equivariant 
and hence it maps the interval $\bar D\cap \R=[-a,a]$ diffeomorphically onto $\bar D'\cap \R=[-b,b]$. 
These improvements are crucial. 
\end{remark}

\begin{proof}
We shall follow the proof of \cite[Lemma 2.1]{ForstnericWold2009} with certain refinements.

By using a biholomorphism $\D\to D$ furnished by Lemma \ref{lem:conformalspecial}
(which extends to a $\Cscr^\infty$ diffeomorphism $\cd\to \overline D$), we see that
it suffices to prove the result when $D=\D$ and hence $a=1$. Choose a smaller open neighborhood
$V_0\subset \C$ of $I=I^+\cup I^-$ such that $\overline V_0\subset V$. Pick a number $\epsilon_0\in(0,\epsilon)$;
it precise value will be specified later. Fix a pair of small discs $U^+_0\Subset U^+_1$ centered at 
the point $1\in\C$, let $U^-_0\Subset U^-_1$ be the corresponding discs centered at the point 
$-1$ given by $U^-_j=\tau_x(U^+_j)$, 
and set $U_j=U^+_j\cup U^-_j$ for $j=0,1$.  (Here, $\tau_x$ and $\tau_y$ are the involutions \eqref{eq:reflections}.)  
We choose these discs small enough such that 
\begin{equation}\label{eq:UinV0}
	U_1 \subset V_0\Subset V.
\end{equation}
Decreasing the number $\epsilon_0>0$ if necessary we may assume that 
\begin{equation}\label{eq:Upm}
	\dist(U_0,\C\setminus U_1) > \frac{\epsilon_0}{2}.
\end{equation}
Fix an integer $n\in\N$ with $n>1/(b-1)$ and let 
\[
	I_n=[1,1+1/n]\cup [-1-1/n,-1].
\] 
Recall that $I=[1,b]\cup [-b,-1]$. Choose a smooth $\C$-valued map $\theta_n$ on a neighborhood of the
compact set
\begin{equation}\label{eq:Kn}
	K_n =\bigl(1+\frac{1}{2n}\bigr)\cd \cup I_n
\end{equation}
which equals the identity on a neighborhood of the closed disc $(1+\frac{1}{2n})\cd$, 
maps the interval $[1,1+\frac{1}{n}]\subset \R$ diffeomorphically onto the interval $[1,b]\subset \R$, 
and satisfies 
\[
	\tau_x \circ \theta_n\circ\tau_x=\theta_n.
\]
In particular, we have that 
\begin{equation}\label{eq:theta-n}
	\theta_n(1+1/n)=b,\quad \theta_n(-1-1/n)=-b.
\end{equation}
By Mergelyan's theorem \cite{Mergelyan1951} we can approximate $\theta_n$ as closely as desired 
in $\Cscr^1(K_n)$ by a polynomial map $\vartheta_n\colon\C\to\C$.
(Mergelyan's theorem provides uniform approximation, but we can apply his result
to the derivative of $\theta_n$ and integrate back in order to get $\Cscr^1$ approximation.)
Furthermore, we can achieve that $\vartheta_n$ satisfies the interpolation conditions \eqref{eq:theta-n}
and also $\vartheta_n(\pm 1)=\pm 1$.  Finally, replacing $\vartheta_n$ by 
$\frac{1}{4}\sum_{\gamma\in \Gamma} \gamma\circ \vartheta_n\circ\gamma$ 
we ensure that $\vartheta_n$ is $\Gamma$-equivariant.
Assuming that $\vartheta_n$ is sufficiently close to  $\theta_n$ in $\Cscr^1(K_n)$, it follows that 
$\vartheta_n$ is biholomorphic in an open neighborhood of $K_n$ \eqref{eq:Kn}. Since $\vartheta_n$ is
$\Gamma$-equivariant, it maps the interval $[1,1+1/n]$ diffeomorphically onto $[1,b]$ and maps
$[-1-1/n,-1]$ diffeomorphically onto $[-b,-1]$. Furthermore, we may assume that 
\begin{equation}\label{eq:estvartheta}
  	|\vartheta_n(z)-z|< \frac{\epsilon_0}{2} \quad \text{for all}\ \ z\in \bigl(1+\frac{1}{2n}\bigr)\cd
	\ \ \text{and}\ n>\frac{1}{b-1},
\end{equation}
and that the $\Gamma$-invariant domain 
\begin{equation}\label{eq:Thetan}
	\Theta_n:= \vartheta_n^{-1}(\D) \Subset \bigl(1+\frac{1}{4n}\bigr)\D
\end{equation}
is an arbitrarily small smooth perturbation of the disc $\D$, with $\pm 1\in b\Theta_n$.

Pick an open neighborhood $W^+_n\subset \C$ of the interval $[1,1+1/n]\subset\R$ and set
$W_n=W^+_n \cup \tau_x(W^+_n)$. By choosing $n$ big and $W^+_n$ small, we may assume that 
\begin{equation}\label{eq:Wn}
	 W_n  \subset  U_0 \quad \text{and}\quad \vartheta_n(W_n)\subset V_0.
\end{equation} 
Let 
\begin{equation}\label{eq:Omegan}
	\Omega_n=\Theta_n\cup R_n 
\end{equation}
be a special domain with the base $(-1-1/n,1+1/n)$, obtained by attaching to $\Theta_n$ 
a thin $\Gamma$-invariant strip $R_n$ around the arcs $[1,1+1/n)\cup (-1-1/n,-1]$ such that 
$R_n \subset W_n$. Together with the second inclusion in \eqref{eq:Wn} we get
\begin{equation}\label{eq:Rn}
	\vartheta_n(R_n)  \subset V_0.
\end{equation}
The (unique) biholomorphic map 
\begin{equation}\label{eq:psin}
	\psi_n\colon \D\to \Omega_n, \quad \psi_n(0)=0, \ \ \psi'_n(0)>0
\end{equation}	
is $\Gamma$-equivariant and satisfies $\psi_n(\pm 1)=\pm(1+1/n)$ by Lemma \ref{lem:conformalspecial}. 
As $n\to\infty$,  the domains $\Theta_n\subset \Omega_n$ converge to the disc $\D$
in the sense of Carath{\'e}odory (the {\em kernel convergence}, see \cite[Theorem 1.8]{Pommerenke1992}),
and their closures $\overline \Theta_n\subset \overline\Omega_n$ also converge to the closed disc $\cd$.
It follows that the sequence of conformal diffeomorphisms $\psi_n\colon \cd\to \overline \Omega_n$ 
converges to the identity map uniformly on $\cd$ by Rado's theorem 
(see Pommerenke \cite[Corollary 2.4, p.\ 22]{Pommerenke1992} 
or Goluzin \cite[Theorem 2, p.\ 59]{Goluzin1969}). In particular, we have that
\begin{equation}\label{eq:estpsi}
	|\psi_n(z)-z|< \frac{\epsilon_0}{2},\quad z\in \cd	
\end{equation}
for all big enough $n\in\N$.  We claim that, for such $n$, the $\Gamma$-invariant domain 
\begin{equation}\label{eq:Dprime}
	D'=\vartheta_n(\Omega_n) 
\end{equation}
and the  $\Gamma$-equivariant biholomorphic map
\begin{equation}\label{eq:phi}
	\phi=\vartheta_n \circ \psi_n \colon \D\to D'
\end{equation}
satisfy the conclusion of the lemma provided that the number $\epsilon_0>0$ is chosen small enough. 
Indeed, condition (a) holds by the construction. We now verify condition (b). 
Firstly, by \eqref{eq:Thetan}, \eqref{eq:Omegan},  and \eqref{eq:psin} we have that
\[
	\D=\vartheta_n(\Theta_n) \subset \vartheta_n(\Omega_n)= 
	\vartheta_n(\psi_n(\D))=\phi(\D)=D'. 
\]
Assume now that $w\in D'\setminus V_0$. By \eqref{eq:Dprime} we have
$w=\vartheta_n(\zeta)$ for some $\zeta\in \Omega_n=\Theta_n\cup R_n$. 
Since $\vartheta_n(R_n)\subset V_0$ by \eqref{eq:Rn} while $w\notin V_0$,
we have $\zeta\in \Theta_n$. By \eqref{eq:Thetan} it follows that $w=\vartheta_n(\zeta)\in\D$.
This shows that 
\begin{equation}\label{eq:D'minusV}
	D'\setminus V_0= \D\setminus V_0
\end{equation}
and hence establishes condition (b) with $V_0$ in place of $V$ (and hence also for $V$).

Finally we verify condition (c). Assume that $z\in \D\setminus U_1$.   
Conditions \eqref{eq:Upm}  and \eqref{eq:estpsi} imply $\psi_n(z)\in \Omega_n\setminus U_0$. 
Since $\Omega_n=\Theta_n\cup R_n$ by \eqref{eq:Omegan} and $R_n \subset W_n\subset U_0$
by \eqref{eq:Wn}, we infer in view of \eqref{eq:Thetan} that 
$\Omega_n\setminus U_0 \subset \Theta_n \subset \bigl(1+\frac{1}{4n}\bigr)\D$
and hence
\[
	\psi_n(z)\in \bigl(1+\frac{1}{4n}\bigr)\cd \quad \text{for all}\ \ z\in \cd\setminus U_1.
\]
By  \eqref{eq:estvartheta}, \eqref{eq:estpsi}, and \eqref{eq:phi} we conclude that
\begin{equation}\label{eq:estphi}
	|\phi(z)-z| \le |\vartheta_n(\psi_n(z)) - \psi_n(z)| + |\psi_n(z)-z| < \frac{\epsilon_0}{2} + \frac{\epsilon_0}{2} =\epsilon_0,
	\quad z\in \cd\setminus U_1. 
\end{equation}
Since $V_0\subset V$ and $\D \setminus V_0 \subset  \D\setminus U_1$ by \eqref{eq:UinV0}, 
this establishes condition (c) for $k=0$.

To complete the proof, we will show that the $\Cscr^k$ estimates of $\phi-\Id$ on 
$\cd\setminus V$ follow from the uniform estimate \eqref{eq:estphi} on $\D\setminus V_0$
in view of the Cauchy estimates and the reflection principle. 
Pick a compact arc $J\subset b\D\setminus \overline V_0$ such that $J\setminus V$ lies in the relative interior of $J$. 
Choose an open neighborhood $E \subset \C$ of $J$ which is invariant with respect to the antiholomorphic 
reflection $\tau(z)=1/\bar z$ around the circle $b\D$ and such that $\overline E\cap \overline{V}_0 =\varnothing$. 
By decreasing $\epsilon_0>0$ if necessary we may assume that $\epsilon_0<\dist(\overline{E},\overline V_0)$. 
It follows from \eqref{eq:estphi} that for every $z\in E\cap \cd$ we have $|\phi(z)-z|<\epsilon_0$ and 
hence $\phi(z)\in \cd\setminus V_0$. In view of \eqref{eq:D'minusV}  it follows in particular that
\begin{equation}\label{eq:phiE}
	\phi(E\cap b\D) \subset b\D\setminus V_0. 
\end{equation}
We extend $\phi$ to $\cd\,\cup E$ by setting 
\[
	\phi(w)=\tau\circ\phi\circ\tau(w),\quad w\in E \setminus \D.
\]
Since $\tau$ fixes the circle $bD$ pointwise, the extended map agrees with $\phi$ on $E\cap b\D$ 
in view of \eqref{eq:phiE}. Since $\tau(w)\in E\cap \cd$ and hence $\phi\circ\tau(w)\in \cd\setminus V_0$ 
by what was said above, the extended map $\phi\colon \cd\, \cup E\to\C$ satisfies the 
estimate $|\phi(z)-z|<C\epsilon_0$ for all $z\in (\cd\setminus V_0) \cup E$, where
the constant $C>1$ depends only on the distortion caused by the reflection $\tau$ on $E$.
Since the compact set $\cd\setminus V$ is contained in the open set $\Omega=(\D\setminus \overline V_0) \cup E $, 
the Cauchy estimates give $\|\phi-\Id\|_{\Cscr^k(\cd\setminus V)} \le C' \|\phi-\Id\|_{\Cscr^0(\Omega)}<C'C\epsilon_0$
for some constant $C'$ depending only on $k$ and $\dist(\cd\setminus V,\C\setminus\Omega)$. 
Choosing $\epsilon_0>0$ small enough, this is $<\epsilon$. The proof is complete.
\end{proof}

An obvious adaptation of the above proof gives the following lemma in which the exposing occurs at a single boundary point.
As mentioned in the introduction, the use of this lemma in the proof of Theorem \ref{th:main}
yields a proper holomorphic embedding $F\colon \D\hra \B^n$ which extends to a holomorphic embedding
of a neighborhood $U\subset\C$ of $\cd\setminus \{1\}$ such that the curve 
$F(b\D\setminus \{1\}) \subset b\B^n$ is dense in $b\B^n$.

%
%
\begin{lemma}\label{lem:special2} 
Let $b>1$. Given an open neighborhood $V\subset \C$ of the interval $I=[1,b]\subset\R$ and 
numbers $\epsilon>0$ and $k\in\Z_+$, there exists a biholomorphism $\phi\colon \D\to \phi(\D)=D'$ onto a 
smoothly bounded domain $D'$ satisfying the following conditions:
\begin{itemize}
\item[\rm (a)] $\phi(0)=0$, $\phi'(0)>0$, $\phi(1)=b$, $\phi(x-\imath y)=\overline{\phi(x+\imath y)}$, 
\vspace{1mm}
\item[\rm (b)] $\D\subset D'\subset \D\cup V$, and 
\vspace{1mm}
\item[\rm (c)] $\|\phi-\Id\|_{\Cscr^k(\overline D\setminus V)}<\epsilon$. 
\end{itemize}
\end{lemma}

Note that $\phi$ extends to a smooth diffeomorphism $\phi\colon\cd\to\bar D'$.
Condition (a) implies that $D'\supset \D\cup [1,b)$ and that $\phi$ maps the interval
$[-1,1]\subset \R$ diffeomorphically onto $[-1,b]$.


\section{Proof of Theorem \ref{th:main}} 
\label{sec:proof}

For simplicity of notation we focus on the case $n=2$
which is of main interest, although the proof holds for any $n\ge 2$.  
 
Let $z=x+\imath y$ denote the coordinate on $\C$. Let $\Gamma$ denote the group \eqref{eq:Gamma}.
Given a positive continuous even function $g>0$ on $\R$, we let $S_g\subset \C$ denote the $\Gamma$-invariant strip
\begin{equation}\label{eq:Sg}
	S_g = \{x+\imath y:  x\in\R,\ \ |y|<g(x)\}.
\end{equation}
Let $f=(f_1,f_2)\colon \R\hra b\B^2$ be a real analytic complex tangential embedding with dense image,
furnished by Corollary \ref{cor:dense}. By complexification, $f$ extends to a holomorphic
immersion $f: \overline S_{g_0} \hra\C^2$ for some $g_0$ as above. 
Fix $g_0$ and write $S=S_{g_0}$. Since the function $|f|^2= |f_1|^2 + |f_2|^2$ is strongly subharmonic on $S$ and
constantly equal to $1$ on $\R$, 
we have $|f(x+\imath y)|\ge 1+c(x)|y|^2$ for a positive smooth function $c\colon \R\to (0,\infty)$
(see e.g.\ \cite{BurnsStout1976} for the details). Hence, if  the strip $S$ is chosen thin enough then 
\begin{equation}\label{eq:norm}
	1\le |f|^2= |f_1|^2 + |f_2|^2 <2 \ \ \text{on} \ \ \overline S,
\end{equation}
and $|f|=1$ holds precisely on $\R$. This means that the immersed complex curve $f(\overline S)\subset \C^2$ 
touches the sphere $b\B^2$ tangentially along $f(\R)$ and satisfies 
\begin{equation}\label{eq:outofball}
	f(\overline S\setminus \R)\subset \sqrt{2}\B^2\setminus \overline\B^2.
\end{equation}

%
%
\begin{lemma}\label{lem:initial-embedding}
Given $\epsilon>0$, there is a strip $S_g$ of the form \eqref{eq:Sg}, 
with $0<g(x)<g_0(x)$ for all $x\in\R$, such that $f\colon \overline S_g \to\C^2$ is an injective immersion and
\[
	\Area(f(S_g)) = \int_{S_g} |f'|^2 dx dy <\epsilon.
\]
\end{lemma}

\begin{proof}
Consider a double sequence $b_j>0$ $(j\in\Z)$ such that 
\[
	0< b_j <\min\{g_0(x):  j-1 \le x\le j\}, \quad j\in \Z.
\]
It follows that the rectangle 
\[
	P_j = \{z=x+\imath y \in \C : j-1 \le x\le j, \ |y|<b_j\} \Subset S
\]
is compactly contained in $S$. We claim that the sequence
$b_j$ can be chosen such that $f$ is injective on the union $\bigcup_{j\in\Z} P_j$.
Indeed, assume that the numbers $b_j$ for $j=0,\pm 1,\ldots, \pm k$ 
have already been chosen such that $f$ is an injective immersion on the set $Q_k=\bigcup_{|j|\le k} P_j$.
In view of \eqref{eq:outofball} if follows that $f$ is an injective immersion on $Q_k\cup \R$;
hence it is an injective immersion in an open neighborhood of the compact set $Q_k\cup [-k-1,k+1]$.
Therefore we can choose the constants $b_{k+1}>0$ and $b_{-k-1}>0$ small enough such that
$f$ is also an injective immersion on $Q_{k+1}$, and hence the induction may proceed.
Finally, choosing a positive even continuous function $g>0$ on $\R$ satisfying 
\[
	\max\{g(x) :  j-1 \le x\le j\} < b_j, \quad j\in \Z,
\]
it follows that $S_g\subset \bigcup_{j\in\Z} P_j= \bigcup_{k\in \N} Q_k$ and hence
$f$ is injective on $S_g$. By decreasing $g$ if necessary we can clearly achieve that 
the area of the disc $f(S_g)$ is as small as desired.
\end{proof}

Replacing the function $g_0$ by $g$ and the initial strip $S=S_{g_0}$ by the strip 
$S_g$ furnished by Lemma \ref{lem:initial-embedding}, we shall assume that 
\begin{equation}\label{eq:2emb}
	\text{$f\colon S\hra \C^2$ is an injective holomorphic immersion.}
\end{equation}

Recall that $\sigma\colon \C^2_*\to\CP^1$ denotes the canonical projection onto the Riemann sphere.
At each point $z=(z_1,z_2)\in b\B^2$ the complex line $\xi_z\subset T_z\C^2$ 
tangent to $b\B^2$ is transverse to the  line $\C z=\sigma^{-1}(\sigma(z))\cup\{0\}$, and hence 
$d\sigma_z \colon \xi_z\to T_{\sigma(z)}\CP^1$ is an isomorphism. Since the immersed curve 
$f\colon \R\to b\B^2$ satisfies $\dot f(t)\in \xi_{f(t)}$ for every $t\in \R$,
it follows that $\sigma \circ f\colon \R\to \CP^1$ is an immersion. 
Hence, if the strip $S$ is chosen thin enough then  
\begin{equation}\label{eq:immersion}
	\sigma \circ f : S\to \CP^1\ \ \text{is an immersion}.
\end{equation}
We shall frequently use the following observation.

\begin{lemma}\label{lem:embed}
Let $D$ be a relatively compact domain in $\C$ and $F\colon \overline D \to \C^2_*$
be an injective immersion of class $\Ascr^1(D)$ such that $\sigma\circ F\colon \overline D\to \CP^1$ is an immersion.
If $h\in H^\infty(D)$ is sufficiently small in the sup-norm, then $e^{-h}F \colon D\to\C^2_*$ is an injective immersion.
\end{lemma}

\begin{proof}
Let $c=\sup\{|F(z)|: z\in\overline D\}>0$. Consider the set
\[
	\Delta=\bigl\{(z,w)\in \overline D\times \overline D: \sigma\circ F(z)=\sigma\circ F(w)\bigr\}
	= \Delta_0\cup \Delta', 
\] 
where $\Delta_0=\{(z,z):z\in\overline D\}$ and $\Delta'=\Delta\setminus \Delta_0$. 
Since $\sigma\circ F\colon \overline D\to\CP^1$ is an immersion, it is locally an embedding, 
and hence the set $\Delta'$ is compact. Since $F\colon \overline D\to \C^2_*$ is injective,  it follows that
\[
	\delta:=\inf\bigl\{|F(z)-F(w)|: (z,w)\in \Delta'\bigr\}>0.
\]
Choose $\mu>0$ such that $|e^\zeta-1|<\delta/3c$ when $|\zeta|<\mu$. Assuming that $|h(z)|<\mu$ for all $z\in D$
and taking into account that $|F|<c$ on $\overline D$, we have for every $(z,w)\in \Delta'$ that 
\begin{eqnarray*}
	\big| e^{-h(z)}F(z)-e^{-h(w)}F(w)\big| &\ge&  |F(z)-F(w)| - \\
	&&  -\bigl| e^{-h(z)}F(z)-F(z)\bigr| - \bigl|e^{-h(w)}F(w)-F(w)\bigr| \\
	&\ge& \delta - c\delta/3c  - c\delta/3c =\delta/3>0. 
\end{eqnarray*}
This shows that the map $e^{-h}F\colon D\to\C^2$ is injective.
Since $\sigma\circ (e^{-h}F) =\sigma \circ F$ is an immersion by the assumption, $e^{-h}F$ is also an immersion.
\end{proof}

We shall find an embedded disc in $\B^2$ satisfying Theorem \ref{th:main} by pulling
the embedded strip $f(S)\subset \C^2_*$ slightly into the ball along the curve $f(\R) \subset b\B^2$,
where the amount of pulling decreases fast enough as we go to infinity inside the strip.
When doing so, we shall pay special attention to ensure injectivity; this is a fairly delicate task
since the curve $f(\R)=f(S)\cap b\B^2$ is dense in $b\B^2$.
To this end, we shall be considering smoothly bounded, simply connected, 
$\Gamma$-invariant domains $D\subset\C$ satisfying
\begin{equation}\label{eq:Sprime}
	\R\subset D\subset \overline {D}\subset S.
\end{equation}
Assume that $h = u+\imath v \in \Ascr^1(D)$ satisfies
\begin{equation}\label{eq:h}
	\text{$u>0$ on $\R$ \ \ and\ \  $e^{2u} < |f|^2$ on $bD$}.
\end{equation}
Consider the map $F$ of class $\Ascr^1(D)$ defined by
\begin{equation}\label{eq:F}
	 F = (F_1,F_2)= e^{-h} (f_1,f_2) : \overline {D} \to \C^2_*. 
\end{equation}
On the real axis we have that $|F|^2=e^{-2u}|f|^2<|f|^2=1$, while on the boundary $bD$ we have
that $|F|^2=e^{-2u}|f|^2>1$ in view of \eqref{eq:h}. This means that
\begin{equation}\label{eq:whatFdoes}
	F(\R)\subset \B^2\quad \text{and}\quad  F(bD)\cap \overline \B^2 =\varnothing.
\end{equation}
Assume in addition that $F$ is transverse to the sphere $b\B^2$. Let 
\begin{equation}\label{eq:Omega}
	\Omega \subset \{z\in D : F(z)\in \B^2\}  
\end{equation}
denote the connected component of the set on the right hand side containing $\R$. It follows that 
$\overline \Omega \subset D$ and $F|_\Omega \colon \Omega\to \B^2$ is a proper holomorphic map
extending holomorphically to $\overline \Omega$ and mapping $b\Omega$ to the sphere $b\B^2$.
Clearly, $\Omega$ is Runge in $\C$ and hence conformally equivalent to the disc;
indeed, there is a biholomorphism $\D\to\Omega$ extending holomorphically
to $\overline \D\setminus \{\pm 1\}$.  The proof of Theorem \ref{th:main} is 
concluded by the following lemma.

%
%
\begin{lemma}\label{lem:main}
Given $\epsilon>0$, there exist a smoothly bounded, simply connected, 
$\Gamma$-invariant domain $D\subset\C$ satisfying \eqref{eq:Sprime}
and a function $h=u+\imath v\in \Ascr^1(D)$ satisfying \eqref{eq:h} such that the 
map $F=e^{-h}f\colon \overline {D} \to \C^2_*$ is an injective immersion 
transverse to $b\B^2$ satisfying
\begin{itemize} 
\item[\rm ($\alpha$)] $\Area(F(\Omega))<\epsilon$, where $\Omega$  is defined by \eqref{eq:Omega}, and
\vspace{1mm}
\item[\rm ($\beta$)]  the curve $F(b\Omega)\subset b\B^2$ is everywhere dense in the sphere $b\B^2$.
\end{itemize}
\end{lemma}

\begin{proof}
We begin by explaining the scheme of proof. 

We shall construct an increasing sequence of  
special domains $D_1\subset D_2\subset D_3\subset \cdots \subset S$ (see Def.\ \ref{def:special}) 
whose union $D=\bigcup_{j=1}^\infty D_j$ is a simply connected, smoothly bounded, $\Gamma$-invariant domain 
satisfying \eqref{eq:Sprime}. The first domain $D_1$ is a round disc centered at $0$; 
by rescaling the coordinate on $\C$ we may assume that $D_1=\d$ is the unit disc. 
For every $n\in \N$ we let $D_{n+1}=D_n\cup S_n$ be a special domain with the base $(-n-1,n+1)\subset \R$,
furnished by Lemma \ref{lem:special}.
(Recall that $S_n$ is a thin strip around the interval $(-n-1,n+1)$.) 
For each $n\ge 2$ let $\psi_n$ be the biholomorphism
\[ 
	\psi_{n}  \colon D_{n}\to D_{n-1},\quad \psi_{n}(0)=0, \ \  \psi'_n(0)>0.
\] 
By Lemma \ref{lem:conformalspecial}, $\psi_{n}$ extends to a smooth $\Gamma$-equivariant diffeomorphism 
$\psi_n\colon \overline D_{n}\to \overline D_{n-1}$ satisfying $\psi_{n}([-n,n])=[-n+1,n-1]$ and 
$\psi_{n}(\pm n)=\pm (n-1)$.  Set 
\begin{equation}\label{eq:Psi-n}
	\Psi_1=\Id|_{\overline D_1},\qquad 
	\Psi_{n}= \psi_2 \circ \cdots \psi_{n}\colon \overline D_n\to \overline D_1 \quad \forall n=2,3,\ldots.
\end{equation}
At the same time, we shall find a sequence of multipliers $h_n\in \Ascr^\infty(D_n)$ $(n\in\N)$ 
of the form $h_n=\tilde h_n\circ\Psi_n$, with $\tilde h_n\in \Hscr^+_*$ (see \eqref{eq:Hplus}), 
such that the map 
\[
	F_n=e^{-h_n}f\colon \overline D_n\hra\C^2_*
\] 
is an embedding of class $\Ascr^\infty(D_n)$ that is transverse to $b\B^2$, 
and $F_{n+1}$ approximates $F_n$ as closely as desired in $\Cscr^0(\overline D_n)$ 
and in $\Cscr^1(\overline D_n\setminus U_n)$, where $U_n=U^+_n\cup U^-_n$ 
is a small neighborhood of the points $\pm n$ for every $n\in\N$.
In the induction step, we shall use Lemma \ref{lem:position} in order to find  a small perturbation of 
$F_n$ such that $F_n(\overline D_n)$ intersects the pair of arcs
$E^+_n=f([n,n+1])\subset b\B^2$ and $E^-_n = f([-n-1,-n])\subset b\B^2$ 
only at the points $f(\pm n)$. This will allow us to construct
the next map $F_{n+1}=e^{-h_{n+1}}f \colon \overline D_{n+1}\hra \C^{2}_*$
which is an embedding mapping the strip $\overline {D_{n+1}\setminus D_n}$
into a small neighborhood of the arcs $E^+_n\cup E^-_n$. 
The sequence $h_n$ will be chosen such that it converges
to a function  $h=u+\imath v  = \tilde h\circ \Psi  \in \Ascr^1(D)$ satisfying \eqref{eq:h}, 
with $\tilde h=\lim_{n\to\infty} \tilde h_n\in \Ascr^1(D_1)$. Furthermore, 
we will ensure that the limit map $F=\lim_{n\to\infty} F_n=e^{-h}f\colon \overline D\to \C^2_*$
(which satisfies \eqref{eq:whatFdoes}  in view of \eqref{eq:h}) 
is an injective immersion that is transverse to the sphere $b\B^2$. 
The domain $\Omega$ \eqref{eq:Omega} will then satisfy the conclusion of the lemma,
and $F(\Omega)\subset\B^2$ will be a properly embedded holomorphic disc satisfying
Theorem \ref{th:main}.

We now turn to the details. Recall that $|f|^2>1$ on $S\setminus \R$.
We begin by choosing a function $h_1=u_1+\imath v_1\in \Hscr^+_*$ on $\overline D_1=\cd$,
close to $0$ in $\Cscr^1(\overline D_1)$, such that 
\begin{equation}\label{eq:u1}
	e^{2u_1} <\frac{1}{2}\left( |f|^2+1\right)\ \ \text{on}\ bD_1\setminus\{\pm 1\}
\end{equation}
and the map $F_1= e^{-h_1}f\colon \overline D_1 \to\C^2_*$
of class $\Ascr^\infty(D_1)$ is an embedding (see Lemma \ref{lem:embed}) 
which is transverse to $b\B^2$ (see Lemma \ref{lem:position}).
From \eqref{eq:u1} we infer that $|F_1|=e^{-u_1}|f|>1$ on $bD_1\setminus \{\pm 1\}$.
Recall  that $|F_1|<1$ on $(-1,1)$ since $u_1>0$ on $D_1$ and $|f|=1$ on $\R$. Let 
\[
	C_1= \{z\in \overline D_1: F_1(z)\in b\B^2\}, \quad \Gamma_1= F_1(C_1)= F_1(\overline D_1)\cap b\B^2.
\]
We have that $C_1 \subset D_1\cup\{\pm 1\}$, each of the sets $C_1$ and $\Gamma_1$ is a union of finitely 
many smooth closed Jordan curves, and $\Gamma_1$ bounds the embedded 
complex curve $F_1(D_1)\cap \B^2$ (see Remark  \ref{rem:Ch}). By \cite{Forstneric1988PM}
the curve $\Gamma_1$ is transverse to the distribution  
$\xi \subset T(b\B^2)$ of complex tangent planes, and hence  
the $\xi$-Legendrian embedding $f\colon \R\hra b\B^2$ is not tangent to $\Gamma_1$ at the points 
$f(\pm 1)\in \Gamma_1$. Thus, there is a number $0<\delta<1$ such that 
\[
	f\bigl([1-\delta,1+\delta] \cup [-1-\delta,-1+\delta]\bigr)  \cap \Gamma_1 = \{f(1),f(-1)\}.
\]
Applying Lemma \ref{lem:position} with the smooth compact curve 
\[
	E=f([-2,-1-\delta]\cup [1+\delta,2]) \subset b\B^2 \setminus  \Gamma_1 
\]
we can approximate $h_1\in \Hscr^+_*$ as closely as desired in the $\Cscr^1(\overline D_1)$ norm 
by a function $\tilde h_1\in \Hscr^+_*$ such that, after redefining the map $F_1$ by setting $F_1=e^{-\tilde h_1}f$
and also redefining the curves $C_1$ and $\Gamma_1$ accordingly, the above conditions still hold
and in addition we have
\begin{equation}\label{eq:avoid}
	f([1,2] \cup [-2,-1]) \cap \Gamma_1 = f([1,2] \cup [-2,-1]) \cap F_1(\overline D_1) =  \{f(1),f(-1)\}.
\end{equation}
Write $\tilde h_1=\tilde u_1+\imath\tilde v_1$. This completes the initial step.

We now explain how to obtain the next embedding $F_2\colon \overline D_2\hra \C^2_*$.
This is the first step of the induction, and all subsequent steps will be of the same kind.

Choose a compact set $M \subset S$ containing $\overline D_1\cup [-2,+2]$ in the interior.
Pick a small number $\mu=\mu_1>0$. By (the proof of) Lemma \ref{lem:embed} we may decrease  $\mu>0$
if necessary such that for any domain $D'\subset M$ and function $h\in \Ascr(D')$ 
satisfying $|h|<\mu$ on $\overline {D'}$ the map $e^{-h}f\colon \overline {D'}\to\C^2_*$
is injective. Since $\tilde u_1$ vanishes on $bD_1=\T$ near $\pm1$ by the definition of the class $\Hscr$,  
there are small discs $U^\pm_1$ around the points $\pm 1$ such that
\begin{equation}\label{eq:tildeu1}
	\text{$\tilde u_1$ vanishes on $bD_1\cap U^\pm_1$}.
\end{equation}
Furthermore, since $\tilde h_1(\pm 1)=0$, we can shrink the discs $U^\pm_1$ if necessary to get
\begin{equation}\label{eq:tildeh1}
	|\tilde h_1|<\mu  \ \ \text{on}\ \ \overline D_1\cap (\overline U^+_1\cup \overline U^-_1)).
\end{equation}
Choose a pair of smaller open discs $W^\pm_1 \Subset V^\pm_1 \Subset U^\pm_1$ around the points $\pm1$.
Lemma \ref{lem:special} furnishes a special domain $D_2$ with the base $(-2,+2)$
satisfying $D_1\subset D_2\subset M$ and a $\Gamma$-equivariant conformal diffeomorphism
$\psi_2=\Psi_2\colon \overline D_{2}\to \overline D_1$ with $\psi_{2}(0)=0$ and $\psi'_{2}(0)>0$.
(By the construction, $D_2$ is the union of $D_1$ and an arbitrarily  thin $\Gamma$-invariant strip  $S_2$
around the interval $(-2,2)\subset\R$, with $\pm 2\in bD_2$.) We may choose $D_2$ such 
that the attaching set $\overline {D_2\cap bD_1}$ is contained in $W^+_1\cup W^-_1$. 
Consider the pair of compact sets
\begin{equation}\label{eq:KL}
	K =\overline D_1 \setminus (W^+_1\cup W^-_1), \quad
	L= \overline{D_2\setminus D_1} \cup (\overline D_1\cap (\overline V^+_1\cup \overline V^-_1)).
\end{equation}
Note that 
\[
	K\cup L=\overline D_2,\quad K\cap L = 
	\overline D_2 \cap \bigl( (\overline V^+_1\setminus W^+_1) \cup (\overline V^-_1\setminus W^-_1) \bigr).
\]
By Lemma \ref{lem:special}, the domain $D_2$ can be chosen such that $\Psi_2$ is as close 
as desired to the identity map in $\Cscr^1(K)$ and 
\begin{equation}\label{eq:Psi2}
	\Psi_2(L) \subset U^+_1\cup U^-_1.
\end{equation}
Set 
\begin{equation}\label{eq:h2}
	h_2=u_2+\imath v_2  := \tilde h_1 \circ \Psi_2 \in \Ascr^\infty(D_2).
\end{equation}
Note that $u_2>0$ on $D_2$, $u_2$ vanishes on $bD_2\setminus \overline D_1$ by \eqref{eq:tildeu1} 
and \eqref{eq:Psi2}, $h_2(\pm 2)=0$, and 
\begin{equation}\label{eq:h2abs}
	|h_2|<\mu\quad \text{on}\ \ \overline{D_2\setminus D_1}
\end{equation}
which follows from  \eqref{eq:tildeh1}, \eqref{eq:KL}, and \eqref{eq:Psi2}.
Assuming that the approximations are close enough, we see from \eqref{eq:u1} that
\begin{equation}\label{eq:eu2}
	e^{2u_2} <\frac{1}{2}\left( |f|^2+1\right)\ \ \text{on}\ bD_2\setminus\{\pm 1\}.
\end{equation}
We claim that the immersion
\[
	F_2=e^{-h_2}f \colon \overline D_2\to \C^2_*
\]
of class $\Ascr^\infty(D_2)$ is injective  provided that the approximations are close enough.
Indeed, $F_2$ is injective on $L$ by the choice of the constant $\mu>0$, 
the estimate \eqref{eq:tildeh1}, the inclusion \eqref{eq:Psi2}, and the definition \eqref{eq:h2}
of $h_2$. Assuming as we may that $\Psi_2$ is close enough to the identity on $K$ (see Lemma  \ref{lem:special}),
the function $h_2|_{K}$ is so close to $\tilde h_1|_{K}$ that $F_2$ is injective on $K$
in view of Lemma \ref{lem:embed}. To obtain injectivity of $F_2$ on $\overline D_2$, it remains to see that  
\[
	F_2(\overline {L\setminus K})\cap F_2(\overline {K\setminus L}) = \varnothing.
\]
Note that $F_2$ maps $\overline {L\setminus K}$ into a small neighborhood of the two arcs 
$f([1,2] \cup [-2,-1])$. Since these arcs intersects 
$F_1(\overline D_1)$ only at the points $F_1(\pm 1) = f(\pm 1)\in \Gamma_1$ (cf.\ \eqref{eq:avoid})
and $F_2$ can be chosen as close as desired to $F_1$ on the set $K$ which does not contain
the points $\pm 1$, the claim follows.

By a slight adjustment of $\tilde h_1$ (and hence of $h_2$, see \eqref{eq:h2}), keeping the above conditions, 
we may assume that the embedding $F_2\colon \overline D_2\hra \C^2$ is transverse to $b\B^2$ (see Lemma \ref{lem:position}). 
Each of the sets $C_2= \{z\in \overline D_2: F_2(z)\in b\B^2\}$ and $\Gamma_2= F_2(C_2)= F_2(\overline D_2)\cap b\B^2$
is then a finite unions of smooth Jordan curves. By another application of
Lemma \ref{lem:position} we can find a $\Cscr^1(\overline D_1)$-small deformation 
$\tilde h_2\in \Hscr^+_*$ of $\tilde h_1$ such that, redefining $h_2$ by setting 
$h_2=\tilde h_2\circ\Psi_2\in \Ascr^\infty(D_2)$ and adjusting the map $F_2$ 
accordingly, the above conditions remain valid and in addition we have that
\[
	f([-3,-2] \cup [2,3]) \cap F_2(\overline D_2)=\{f(2),f(-2)\}.
\]
This completes the first step of the induction, and we are now ready to apply the same arguments 
to the map $F_2$ and the domain $D_2$ to find the next embedding $F_3\colon \overline D_3\hra \C^2_*$.

Clearly this construction can be continued inductively. It yields 
\begin{itemize}
\item an increasing sequence of special domains $D_1\subset D_2\subset D_3\subset\cdots$ such that
$D_n$ has the base $(-n,n)$, $D_{n+1}=D_n\cup S_n$ where $S_n$ is a thin strip around the interval $(-n-1,n+1)$,
and the union $\overline {D}=\bigcup_{n=1}^\infty \overline D_n\subset S$ is a smoothly bounded, simply 
connected, $\Gamma$-invariant domain, 
\vspace{1mm}
\item a sequence $U_{n}=U^-_n\cup U^+_n$ of small pairwise disjoint neighborhood of the points $\{-n,n\}$ 
such that $\overline D_{n-1}\cap \overline U_n=\varnothing$ for every $n=2,3,\ldots$, 
\vspace{1mm}
\item a sequence of $\Gamma$-equivariant diffeomorphisms 
$\Psi_n= \psi_2 \circ \cdots \psi_{n}\colon \overline D_n\to \overline D_1$
of class $\Ascr^{\infty}(D_n)$, where $\Psi_1=\Id|_{\overline D_1}$ and 
$\psi_n \colon \overline D_n\to \overline D_{n-1}$ for $n\ge 2$ (see  \eqref{eq:Psi-n}), 
\vspace{1mm}
\item a sequence of multipliers
\begin{equation}\label{eq:hn}
	h_{n}=\tilde h_n\circ\Psi_n \in \Ascr^\infty(D_{n}) \quad \text{with}\ \  \tilde h_n\in \Hscr^+_*,
\end{equation}	
\item a sequence of embeddings 
$
	F_n=e^{-h_n}f\colon \overline D_n\hra \C^2_*
$
of class $\Ascr^\infty(D_n)$,
\end{itemize}
such that the following conditions hold for every $n\in\N$:
\begin{itemize}
\item[\rm (a$_n$)]  the conformal diffeomorphism $\psi_{n+1}\colon \overline D_{n+1}\to\overline D_n$ 
is arbitrarily close to the identity in $\Cscr^1(\overline D_{n}\setminus U_n)$ 
and satisfies $\psi_{n+1}(\overline{D_{n+1}\setminus D_{n}}) \subset U_n$;
\vspace{1mm}
\item[\rm (b$_n$)] the function $h_n=u_n+\imath v_n$ \eqref{eq:hn} satisfies $u_n>0$ on $D_n$ and
$e^{2u_n} < \frac{1}{2}(|f|^2+1)$ on $bD_n\setminus \{\pm n\}$  (see \eqref{eq:eu2}); 
\vspace{1mm}
\item[\rm (c$_n$)] $\tilde h_{n+1}$ approximates $\tilde h_n$  as closely as desired in $\Cscr^1(\overline D_1)$;
\vspace{1mm}
\item[\rm (d$_n$)] $h_{n+1}$ approximates $h_n$  as closely as desired in $\Cscr^0(\overline D_n)$
and in $\Cscr^1(\overline D_n\setminus U_n)$; 
\vspace{1mm}
\item[\rm (e$_n$)] 
$|h_{n+1}|$ is  as small as desired uniformly  on $\overline {D_{n+1}\setminus D_n}$ (see \eqref{eq:h2abs}); 
\vspace{1mm}
\item[\rm (f$_n$)] the map $F_{n+1}$ approximates $F_{n}$ as closely as desired uniformly on $\overline D_n$
and in $\Cscr^1(\overline D_n\setminus U_n)$,  and $F_{n+1}$ is as close as desired to $f$ 
uniformly on $\overline {D_{n+1}\setminus D_n}$. 
\end{itemize}

Note that (f$_n$) is a consequence of (a$_n$), (c$_n$), (d$_n$), and (e$_n$).

Assuming that these approximations are close enough at every step, we can draw the following conclusions.
The sequence $\Psi_n\colon\overline D_n\to \overline D_1$ 
converges to the $\Gamma$-equivariant biholomorphic map 
$\Psi \colon  D=\bigcup_{n=1}^\infty D_n \to \D$ with $\Psi(0)=0$ and $\Psi'(0)>0$.
Note that $\Psi(\R)=(-1,1)$, and $\Psi$ extends to a $\Cscr^\infty$ diffeomorphism $\overline {D} \to \cd\setminus \{\pm1\}$.
Secondly, the sequence $h_n\in\Ascr^\infty(D_n)$ converges in the weak  $\Cscr^1(\overline {D})$ topology
(i.e., in the $\Cscr^1$ topology on every compact subset of $\overline D$) to a function 
\begin{equation}\label{eq:tildeh}
	h=u+\imath v= \tilde h \circ\Psi  \in \Ascr^1(D) \ \ \ \text{where}\ \ \ \tilde h=\lim_{n\to \infty} \tilde h_n \in \Ascr^1(D_1)
\end{equation}
(the second limit $\tilde h$ exists in the $\Cscr^1(\overline D_1)$ topology in view of condition (c$_n$)) 
satisfying 
\begin{equation}\label{eq:limitestimate}
		u>0\ \ \text{on}\ \ D, \quad   e^{2u} \le \frac{1}{2}\left( |f|^2+1\right) < |f|^2\ \ \text{on}\ bD. 
\end{equation}
Thirdly, the sequence of embeddings $F_n= e^{-h_n}f \colon \overline D_n\hra \C^2_*$ converges 
in the weak  $\Cscr^1(\overline {D})$ topology to the map $F=e^{-h}f \colon \overline {D}\to\C^2$ of class $\Ascr^1(D)$.
Since $\sigma\circ F = \sigma\circ f\colon \overline D\to \CP^1$ is an immersion, $F$ is an immersion on $\overline D$.
Lemma \ref{lem:embed} shows that $F$ is injective (hence an embedding) on every domain $\overline D_n$,
and therefore also on $\overline D$, provided that the approximation of $F_n$ by $F_{n+1}$  
is close enough in $\Cscr^0(\overline D_n)$ for every $n\in\N$ (see (f$_n$)).
Note that conditions \eqref{eq:h} holds in view of \eqref{eq:limitestimate}, and hence 
$F$ satisfies condition \eqref{eq:whatFdoes}. It follows that the simply connected domain $\Omega$ 
\eqref{eq:Omega} (with $\R\subset \Omega\subset\overline\Omega\subset D$)
is well defined,  and the restricted map $F|_\Omega \colon \Omega\hra\B^2$ is a properly embedded holomorphic disc. 
Since every map $F_n\colon \overline D_n\to \C^2$ is transverse to $b\B^2$,
the same is true for $F\colon \overline D\to \C^2$ provided the approximation is close enough at every step.

It remains to show conditions ($\alpha$) and ($\beta$) in the lemma.

Given a sequence  $\epsilon_1>\epsilon_2>\cdots>0$ with $\lim_{n\to\infty} \epsilon_n=0$,
conditions (d$_n$) and (e$_n$) show that the sequence $h_{n}$ \eqref{eq:hn} can be chosen
such that its limit $h$ \eqref{eq:tildeh} satisfies
\[
	|h|<\epsilon_n \quad \text{on}\ \ \overline{D_{n+1}\setminus D_n}\ \ \ \text{for all}\ \ n\in \N.
\]
This implies that $F(b\Omega)\subset b\B^2$ consists of a pair of curves in the sphere
which are as close as desired to the curve $f(\R)\subset b\B^2$ in the fine $\Cscr^0$ topology.
Since $f(\R)$ is dense in $b\B^2$, the same is true for $F(b\Omega)$ provided the
approximation is close enough, so condition ($\beta$) holds. 
(This argument is similar to the one in \cite[proof of Theorem VI.1]{GlobevnikStout1986}.)

It remains to estimate the area of the disc $F(\Omega)\subset \B^2$. Set 
\begin{equation}\label{eq:c}
	c = \max \left\{ |\tilde h'(z)| : z\in\overline D_1\right\}
\end{equation}
where $\tilde h \in \Ascr^1(D_1)$ is as in \eqref{eq:tildeh}.
Note that $c>0$ can be made as small as desired by choosing all terms of the sequence 
$\tilde h_n\in \Hscr^+_*$ small enough in the $\Cscr^1(\overline D_1)$ norm. 
Recall that $h=\tilde h \circ\Psi$ (see \eqref{eq:tildeh}). We have $h'(z) = \tilde h'(\Psi(z)) \Psi'(z)$ and hence
$|h'(z)|^2 \le c^2 |\Psi'(z)|^2$ for $z\in D$. Differentiation of $F=e^{-h}f$ gives $F'=-e^{-h}h'f+e^{-h}f'$.
Note that $u>0$ and hence $e^{-u}<1$ on $D$. Recall also that $|f|^2\le 2$ on $S$ (see \eqref{eq:norm}). 
Using the inequality $|a+b|^2\le 2(|a|^2+|b|^2)$, which holds for any $a,b\in\C^n$, we thus obtain
\begin{eqnarray*}
	\Area(F(\Omega)) &=& \int_{\Omega} |F'|^2 dxdy  \le  
	2\int_{\Omega} e^{-2u}|h'|^2 |f|^2 dxdy + 2 \int_{\Omega} e^{-2u}  |f'|^2 dxdy\\
		&\le & 
		4c^2 \int_{\Omega}  |\Psi'(z)|^2 dxdy +  2 \int_{\Omega} |f'|^2 dxdy \\
		& = & 4c^2 \Area(\Psi(\Omega)) + 2\Area(f(\Omega)).
\end{eqnarray*}
Since $\Psi(\Omega)\subset \Psi(D)=\D$, we have $\Area(\Psi(\Omega)) \le \pi$,
and hence the first term is bounded by $\epsilon/2$ if $c>0$ is small enough.
Lemma \ref{lem:initial-embedding} shows that the second term can be chosen $<\epsilon/2$.
This completes the proof of Lemma \ref{lem:main},  and hence of Theorem \ref{th:main}.
\end{proof}


\subsection*{Acknowledgements}
This research was supported  in part by the research program P1-0291 and grant J1-7256 from 
ARRS, Republic of Slovenia. I wish to thank Filippo Bracci for having brought  to my attention the question 
answered in the paper, Bo Berndtsson for the communication regarding the reference \cite{Berndtsson1980},
Josip Globevnik for helpful discussions and the reference to his work  \cite{GlobevnikStout1986}
with E.\ L.\ Stout, Finnur L{\'a}russon for remarks concerning the exposition, and Erlend F.\ Wold
for the communication on Sect.\ \ref{sec:conformal}.



\begin{thebibliography}{100}

\bibitem{Abraham1963}
R.~Abraham.
\newblock Transversality in manifolds of mappings.
\newblock {\em Bull. Amer. Math. Soc.}, 69:470--474, 1963.

\bibitem{AlarconForstnericLopez2017CM}
A.~Alarc{\'o}n, F.~Forstneri{\v c}, and F.~J. L{\'o}pez.
\newblock Holomorphic {L}egendrian curves.
\newblock {\em Compos. Math.}, 153(9):1945--1986, 2017.

\bibitem{AlarconGlobevnikLopez2016Crelle}
A.~{Alarc{\'o}n}, J.~{Globevnik}, and F.~J. {L{\'o}pez}.
\newblock {A construction of complete complex hypersurfaces in the ball with
  control on the topology}.
\newblock {\em J. Reine Angew. Math., {\rm to appear}}.
\newblock
  \url{https://www.degruyter.com/view/j/crll.ahead-of-print/crelle-2016-0061/crelle-2016-0061.xml}.

\bibitem{Berndtsson1980}
B.~Berndtsson.
\newblock Integral formulas for the {$\partial \bar \partial $}-equation and
  zeros of bounded holomorphic functions in the unit ball.
\newblock {\em Math. Ann.}, 249(2):163--176, 1980.

\bibitem{BurnsStout1976}
D.~Burns, Jr. and E.~L. Stout.
\newblock Extending functions from submanifolds of the boundary.
\newblock {\em Duke Math. J.}, 43(2):391--404, 1976.

\bibitem{Caratheodory1913}
C.~{Carath{\'e}odory}.
\newblock {\"Uber die Begrenzung einfach zusammenh\"angender Gebiete.}
\newblock {\em {Math. Ann.}}, 73:323--370, 1913.

\bibitem{CieliebakEliashberg2012}
K.~Cieliebak and Y.~Eliashberg.
\newblock {\em From {S}tein to {W}einstein and back. Symplectic geometry of
  affine complex manifolds}, volume~59 of {\em Amer. Math. Soc. Colloquium
  Publications}.
\newblock Amer. Math. Soc., Providence, RI, 2012.

\bibitem{DiederichFornaessWold2014}
K.~Diederich, J.~E. Forn{\ae}ss, and E.~F. Wold.
\newblock Exposing points on the boundary of a strictly pseudoconvex or a
  locally convexifiable domain of finite 1-type.
\newblock {\em J. Geom. Anal.}, 24(4):2124--2134, 2014.

\bibitem{Forstneric1988PM}
F.~Forstneri{\v c}.
\newblock Regularity of varieties in strictly pseudoconvex domains.
\newblock {\em Publ. Mat.}, 32(1):145--150, 1988.

\bibitem{ForstnericWold2009}
F.~Forstneri{\v{c}} and E.~F. Wold.
\newblock Bordered {R}iemann surfaces in {${\mathbb C}^2$}.
\newblock {\em J. Math. Pures Appl. (9)}, 91(1):100--114, 2009.

\bibitem{Geiges2008}
H.~Geiges.
\newblock {\em An introduction to contact topology}, volume 109 of {\em
  Cambridge Studies in Advanced Mathematics}.
\newblock Cambridge University Press, Cambridge, 2008.

\bibitem{GlobevnikStout1986}
J.~Globevnik and E.~L. Stout.
\newblock The ends of varieties.
\newblock {\em Amer. J. Math.}, 108(6):1355--1410, 1986.

\bibitem{GolubitskyGuillemin1973}
M.~Golubitsky and V.~Guillemin.
\newblock {\em Stable mappings and their singularities}, volume~14 of {\em
  Graduate Texts in Mathematics}.
\newblock Springer-Verlag, New York-Heidelberg, 1973.

\bibitem{Goluzin1969}
G.~M. Goluzin.
\newblock {\em Geometric theory of functions of a complex variable}.
\newblock Translations of Mathematical Monographs, Vol. 26. American
  Mathematical Society, Providence, R.I., 1969.

\bibitem{Gromov1996PM}
M.~Gromov.
\newblock Carnot-{C}arath\'eodory spaces seen from within.
\newblock In {\em Sub-{R}iemannian geometry}, volume 144 of {\em Progr. Math.},
  pages 79--323. Birkh\"auser, Basel, 1996.

\bibitem{Kellogg1912}
O.~D. Kellogg.
\newblock Harmonic functions and {G}reen's integral.
\newblock {\em Trans. Amer. Math. Soc.}, 13(1):109--132, 1912.

\bibitem{KoziarzMok2010}
V.~Koziarz and N.~Mok.
\newblock Nonexistence of holomorphic submersions between complex unit balls
  equivariant with respect to a lattice and their generalizations.
\newblock {\em Amer. J. Math.}, 132(5):1347--1363, 2010.

\bibitem{Mergelyan1951}
S.~N. Mergelyan.
\newblock On the representation of functions by series of polynomials on closed
  sets.
\newblock {\em Doklady Akad. Nauk SSSR (N.S.)}, 78:405--408, 1951.

\bibitem{Narasimhan1985}
R.~Narasimhan.
\newblock {\em Analysis on real and complex manifolds}, volume~35 of {\em
  North-Holland Mathematical Library}.
\newblock North-Holland Publishing Co., Amsterdam, 1985.
\newblock Reprint of the 1973 edition.

\bibitem{Pommerenke1992}
C.~Pommerenke.
\newblock {\em Boundary behaviour of conformal maps}, volume 299 of {\em
  Grundlehren Math. Wiss.}
\newblock Springer-Verlag, Berlin, 1992.

\bibitem{Whitney1936}
H.~Whitney.
\newblock Differentiable manifolds.
\newblock {\em Ann. of Math. (2)}, 37(3):645--680, 1936.

\bibitem{Zhang2017}
L.~Zhang.
\newblock Intrinsic derivative, curvature estimates and squeezing function.
\newblock {\em Sci. China Math.}, 60(6):1149--1162, 2017.

\end{thebibliography}

\def\cprime{$'$}


\vspace*{0.5cm}
\noindent Franc Forstneri\v c

\noindent Faculty of Mathematics and Physics, University of Ljubljana, Jadranska 19, SI--1000 Ljubljana, Slovenia.

\noindent Institute of Mathematics, Physics and Mechanics, Jadranska 19, SI--1000 Ljubljana, Slovenia.

\noindent e-mail: {\tt franc.forstneric@fmf.uni-lj.si}

\end{document}